\documentclass[12pt, leqno]{amsart}
\usepackage{amssymb,amsmath,amsthm,nomencl,mathrsfs,verbatim} 
\usepackage[arrow, matrix, curve]{xy}
\usepackage[latin1]{inputenc}
\usepackage{a4wide}
\newcommand{\IC}{\mathbb{C}}
\newcommand{\IR}{\mathbb{R}}

\newcommand{\IL}{\mathsf{L}}

\newcommand{\ILL}{\mathscr{L}}

\newcommand{\IPP}{\mathscr{P}}

\newcommand{\IAA}{\mathscr{A}}

\newcommand{\IN}{\mathbb{N}}
\newcommand{\IZ}{\mathbb{Z}}

\newcommand{\pa}{\slash\slash}

\newcommand{\sm}{\sim_b}
\newcommand{\Id}{{\rm d}}

\newcommand{\f}{\frac}
\newcommand{\nn}{\nonumber}

\theoremstyle{plain}            % body italics
\newtheorem{theorem}{theorem}[section]
\newtheorem{Lemma}[theorem]{Lemma}
\newtheorem{Corollary}[theorem]{Corollary}
\newtheorem{Theorem}[theorem]{Theorem}
\newtheorem{Proposition}[theorem]{Proposition}
\newtheorem{Propandef}[theorem]{Proposition and Definition}

\newtheorem{Theoreme}{Theorem}

\theoremstyle{definition}       % body roman
\newtheorem{Definition}[theorem]{Definition}

\newtheorem{Remark}[theorem]{Remark}
\newtheorem{Example}[theorem]{Example}

\begin{document}

\begin{titlepage}
\title[Heat kernels and relative compactness]
{Heat kernel estimates and the relative compactness of perturbations by 
potentials}

\clearpage
\setcounter{page}{1}

\author[J. Brüning]{Jochen Brüning}
\address{Jochen Brüning, Institut f\"ur Mathematik, Humboldt-Universit\"at 
zu Berlin, 12489 Berlin, Germany} \email{bruening@math.hu-berlin.de}

\author[B. G\"uneysu]{Batu G\"uneysu}
\address{Batu G\"uneysu, Institut f\"ur Mathematik, Humboldt-Universit\"at 
zu Berlin, 12489 Berlin, Germany} \email{gueneysu@math.hu-berlin.de}

\end{titlepage}

\maketitle

\begin{center}
\today

\end{center}
\begin{abstract} We consider a self-adjoint non-negative operator $H$ in a Hilbert space 
$\mathsf{L}^2(X,{\rm d}\mu)$. We assume that the semigroup $(\mathrm{e}^{-t H})_{t>0}$ is defined by an integral kernel, $p$, which allows an estimate of the form $p(t,x,x)\le F_1(x)F_2(t)$ for all $(x,t)\in X\times\mathbb{R_+}$; we refer to $F_1$ as the \emph{control function}. We show that such an estimate leads to rather satisfying abstract results on relative compactness of perturbations of $H$ by potentials. It came as a surprise to us, however, that such an estimate holds for the Laplace-Beltrami operator on \emph{any} Riemannian manifold. In particular, using a domination principle, one can deduce from the latter fact a very general result on the relative compactness of perturbations by potentials of the Bochner Laplacian associated with a Hermitian bundle 
$(E, h^E,\nabla^E)$ over an arbitrary Riemannian manifold $(M,g)$; in fact, only quantities of  order zero in $g$ enter in the estimates. We extend this result to 
weighted Riemannian manifolds, where under lower curvature bounds on the 
$\alpha$-Bakry-\'{E}mery tensor one can construct quite explicit control functions, and to any weighted graph, where the control function is expressed in terms of the vertex weight function.    
\end{abstract}

\tableofcontents

\section{Introduction}
Consider the Schr\"odinger operator $ - \Delta + V$ in $\mathbb{R}^3$. It is well 
known that $V(- \Delta + 1)^{-1}$ is compact if we have a decomposition $V = V_1 + V_2$ such that $V_1\in\mathsf{L}^2(\IR^3)$ and $V_2$ is bounded and vanishes at $\infty$ 
(cf. \cite[Example 6, p. 117]{Re4}, and also \cite[Section 11.2]{davies} for 
further Euclidean results). Hence $V$ is a relatively compact perturbation of $- \Delta $, 
and the perturbation preserves the domain, the self-adjointness, and the essential 
spectrum; this result applies in particular to the hydrogen atom.\\
In the case of a non-parabolic three-dimensional Riemannian manifold $M$, the 
analogous result for the Laplace-Beltrami operator with the potential 
$V(x) := -\kappa G(x,x_0), x_0\in M$ arbitrary, $\kappa\ge 0$, and 
$$
G(x,x_0)=\int^{\infty}_0 {\rm e}^{t\Delta}(x,x_0) \Id t
$$
the minimal non-negative Green's function for (the Friedrichs realization of) 
$-\Delta\geq 0$, has been proved in \cite{enciso} only under additional geometric assumptions, 
namely geodesic completeness, non-negativity of the Ricci curvature, and a 
Euclidean volume growth of geodesic balls $B(x,r)$ from below. These assumptions lead to a heat kernel estimate of the form
\begin{align}\label{hupa}
\mathrm{e}^{t\Delta}(x,x)\leq C t^{-3/2}\>\>\text{ for all $t>0$},
\end{align}
which is used heavily in the proof given in \cite{enciso} of the asserted relative compactness, not only for the existence of $G(x,x_0)$, but as a property of the underlying \lq\lq{}free operator\rq\rq{}.\\
\newline
In this paper we are interested in general relative compactness results for (certain generalizations of) the Laplace-Beltrami operator on Riemannian manifolds, in particular, in results that do not require too much of a Euclidean behaviour of the geometry from the beginning. Here, we are particularly interested in situations where the \lq\lq{}free operator\rq\rq{} $H$ is actually itself some kind of perturbation of the Laplace-Beltrami operator, such as a magnetic Schrödinger operator  
$$
H=(\Id+\sqrt{-1}\beta)^{\dagger}(\Id+\sqrt{-1}\beta)+V\rq{}
$$
with $\beta$ a magnetic potential and $V\rq{}:M\to\IR$ a potential, and the perturbation of $H$ is given by a potential $V:M\to \IR$. In the latter situation, a very general compactness result for the operator $V(H+1)^{-1}$ will in fact follow easily from some of our main results (cf. Example \ref{wasser}).
\newline

In order to explain our main results, let now $M$ be an arbitrary Riemannian manifold. The central observation of this paper are the following two heat kernel estimates: The first one holds on {\em any} Riemannian manifold whatsoever, and involves only a zero order 
geometric quantity which we call the Euclidean radius at $x\in M$ with 
distortion $b > 1$, to be denoted $r_{\rm Eucl}(x,b)$ (see Def. \ref{deee}). The second heat kernel estimate is one for complete Riemannian manifolds with nonnegative Ricci curvature, and follows from the well-known Li-Yau heat kernel estimate:

\begin{Theoreme}\label{bou} a) There exists a universal constant $C>0$, which only depends on $m=\dim(M)$, such that for every $(t,x)\in (0,\infty)\times M$ one has
$$
\mathrm{e}^{t\Delta}(x,x)\leq C \; \min (r_{\rm Eucl}(x,2),1)^{-m} \big(t^{-m/2} +1\big)=:F_1(x)\cdot F_2(t).
$$
b) If $M$ is complete with $\mathrm{Ric}\geq 0$, then there exists a universal constant $C>0$, which only depends on $m=\dim(M)$, such that for every $(t,x)\in (0,\infty)\times M$ one has (cf. Section \ref{dhjkp})
$$
\mathrm{e}^{t\Delta}(x,x)\leq   C \; \mathrm{vol}\big(B(x,1)\big)^{-1}\cdot\big(t^{-m/2} +1\big)=:F_1(x)\cdot F_2(t).
$$
\end{Theoreme}

Theorem \ref{bou}.a) (which does not even require completeness) is a straightforward consequence of a heat kernel estimate which is essentially due to A. Grigor'yan and which follows from the 
parabolic $\IL^2-$mean value inequality (cf. \cite[Thm. 15.14] {buch}). We give a detailed proof with explicit constants of this estimate in the context of weighted Riemannian manifolds in Theorem \ref{gri}. Furthermore, we will show in  Proposition \label{below} that $r_{\rm Eucl}(x,b)$ can be estimated from below by familiar geometric quantities. \\
Concerning Theorem \ref{bou}.b), we will in fact prove a more general variant of the latter estimate (cf. in Theorem \ref{ghhp}) in the context of weighted Riemannian manifolds that have nonnegative $\alpha$-Bakry-\'{E}mery tensor for some $\alpha>0$ (cf. Definition \ref{bakry}), where in this general case the volume is replaced with the weighted volume, and more importantly, the function $t^{-m/2}$ has to be replaced with $t^{-(m/2+\alpha/2)}$, which has a strong influence on the results below.\\
The above estimates now lead to the conjecture that good conditions for relative 
compactness of a potential could be expressed in terms of the "control" functions 
$F_1$ and $F_2$. This can, in fact, be done in a purely measure theoretic setting, and constitutes the first main result, Theorem \ref{ddp} below, of 
this work; it will be proved in Thm. \ref{main} below. Consider then an arbitrary sigma-finite measure space, 
$(X,\mu)$, and let 
\begin{align*}
p:(0,\infty)\times X\times X\longrightarrow 
[0,\infty),\>\>(t,x,y)\longmapsto  p(t,x,y)  =: p(t,x,y),
\end{align*}
be a {\em pointwise consistent $\mu$-heat kernel} (see Definition \ref{heat} 
below for details). Any 
such function $p$ canonically induces a strongly continuous 
contraction semigroup $(P_t)_{t\geq 0}$ the generator of which will be 
denoted by $H_p$.
\begin{Theoreme}\label{ddp}
Assume that there exist functions $F_1$ on $X$ and $F_2$ on 
on $(0,\infty)$ such that for all $t>0$, $x\in X$, and for some $q \geq  1$
\begin{align}\label{ffop}
&p(t,x,x)\leq F_1(x)F_2(t),\text{ }\\
%\nn
&\int^{\infty}_{0}\mathrm{e}^{-t}F_2(t)^{\f{1}{2q}}\Id t<\infty.
\end{align} 
Then for any potential $V:X\to \mathbb{R}$ of the form $V=V_1+V_2$, the 
operator 
$V(H_p+1)^{-1}$ is compact if one of the following two conditions hold:

\begin{align}
&q=1,\,\,\, V_1\in\IL^2(X, F_1{\rm d}\mu), \,\,\,
V_2\in\IL^{\infty}_{\infty}(X,{\rm d}\mu)\cap \IL_{\infty}(X,F_1{\rm d}\mu),\\ 
%vanishing at $\infty$ (cf. Def. \ref{van}), with respect to $\Id\mu$ and $F_1\Id\mu$\\
&q>1,\,\,\, F_1\equiv 1,\,\,\, V_1\in\IL^{2q}(X,{\rm d}\mu), \,\,\,
V_2\in\IL^{\infty}_{\infty}(X,{\rm d}\mu). 
%vanishing at $\infty$,  with respect to $\Id\mu$.
\end{align}
Here, $\IL^{\infty}_{\infty}(X,\cdot)$ denotes the space of functions 
in $\IL^{\infty}(X,\cdot)$ which vanish at infinity w.r.t. the 
relevant measure, as defined in Def. 2.3.
\end{Theoreme}

In fact, there is a more general version of this result which holds on 
measurable vector bundles $E\to X$ (cf. Cor. \ref{mag1}), There, 
we can replace $H_p$ by an arbitrary self-adjoint operator $\tilde{H}\geq 0$ on a 
space of $\IL^2$-sections $\Gamma_{\IL^2}(M,E;\Id\mu)$ which satisfies a Kato-type 
domination property $\tilde{H}\succeq H_p$, and the potential is understood to be 
a self-adjoint section in $\mathrm{End}(E)\to M$.\\
\newline
Ultimately, the general vector bundle variant of Theorem \ref{ddp} can be brought into the following form for Riemannian manifolds, where, in this introduction we restrict ourselves to $\dim(M)\leq 3$, noting that the (more technical) results for higher dimensions can be found in Corollary \ref{has} and Corollary \ref{has2}.

\begin{Theoreme}\label{haes} Let $M$ be a  Riemannian manifold with dimension $m\leq 3$, let $\nabla$ be a unitary covariant derivative on the Hermitian vector bundle $E\to M$, and let $0\leq V\in\Gamma_{\mathsf{L}^1_{\mathrm{loc}}}(M,\mathrm{End}(E))$ be a potential. We denote the Friedrichs realization of $\nabla^{\dagger}\nabla$ in $\Gamma_{\mathsf{L}^2}(M,E)$ with the same symbol again, and with $\nabla^{\dagger}\nabla+V$ the corresponding form sum, and we assume that $W\in\Gamma(M,\mathrm{End}(E))$ is a potential of the form $W=W_1+W_2$ with $W_2$ bounded.\\
a) If there exists $b>1$ such that for all $c>0$ one has 
\begin{align*} 
\int  \big(|W_1|(x)^{2} +1_{\{|W_2|>c\}}(x)\big)    \min(r_{\mathrm{Eucl}}(x,b),\epsilon)^{-m}     \mathrm{vol}(\Id x) <\infty,
\end{align*}
then $W(\nabla^{\dagger}\nabla+V+1)^{-1}$ is a compact operator in $\Gamma_{\mathsf{L}^2}(M,E)$.\\
b) If $M$ is complete with $\mathrm{Ric}\geq 0$, and if for all $c>0$ one has 
\begin{align*} 
\int  \big(|W_1 |(x)^{2}  +1_{\{|W_2|>c\}}(x)  \big)     \mathrm{vol} \big(B_g(x,1)\big)^{-1} \mathrm{vol}(\Id x) <\infty,
\end{align*}
then again $W(\nabla^{\dagger}\nabla+V+1)^{-1}$ is a compact operator in $\Gamma_{\mathsf{L}^2}(M,E)$.
  \end{Theoreme}

Note that Theorem \ref{haes}.a) does not require any further assumptions on the Riemannian manifold. The point of part b) is that in case one actually has a lower control on the Ricci curvature, one can simply pick a somewhat more explicit control function. We explain in Example \ref{wasser} how one can handle hydrogen type problems (even in the presence of a magnetic field and an additional positive potential) on curved space within the setting of Theorem \ref{haes}.b). These observations clarify, in particular, that a lower Euclidean volume growth assumption is actually only required to get a well-behaved Coulomb potential at all, and not to get general results of the form \lq\lq{}$W$ is integrable in some (weighted) sense $\Rightarrow$ $W(-\Delta+1)^{-1}$ is compact\rq\rq{}.\\
We will also provide a variant of Theorem \ref{haes}.b) and its \lq\lq{}high-dimensional\rq\rq{} version for manifolds with nonnegative $\alpha$-Bakry-\'{E}mery tensor for some $\alpha>0$, where, in addition to obvious modifications, this generalization has the effect of replacing $m\leq 3$ with $m+\alpha\leq 3$. This is the content of Corollary \ref{has3}. As in many other applications of such weighted Riemannian manifolds, this leads to an interpretation of $\alpha$ as an additional \lq\lq{}virtual dimension\rq\rq{} of the underlying space (cf. Remark \ref{michel}).

\vspace{2mm}

Finally, we mention that Theorem \ref{ddp} and its covariant version can 
also be applied to {every (possibly locally infinite) graph} $X$ that 
carries weights on its edges and its vertices, as the corresponding heat 
kernel $p(t,x,y)$ always satisfies an upper bound of the form
\begin{align*}
p(t,x,x)\leq 1/\varrho(x)=:F_1(x),\text{ with $\varrho:X\to (0,\infty)$ the 
vertex weight function,}
\end{align*}
Here, the latter bound follows intuitively from observing that $p(t,x,x)\varrho(x)$ 
is nothing but the probability of finding the underlying Markoff particle in 
$x\in X$ at the time $t$, when conditioned to start in $x$. The corresponding 
relative compactness result is formulated in Theorem \ref{grap}.\vspace{8mm}

{\bf Conventions:} Given a measure space $(X,\IAA,\mu)$, we will omit the 
underlying sigma-algebra $\IAA$ in the notation, and simply speak of 
\lq\lq{}measurable sets\rq\rq{}, whenever there is no danger of confusion. The 
$\mathsf{L}^q$-norm corresponding to the measure space $(X,\mu) := 
(X,\IAA,\mu)$ will be denoted by 
$\left\|\cdot\right\|_{\mu,q}$, $q\in [1,\infty]$, and the corresponding 
operator norm of bounded operators 
$\mathsf{L}^{\alpha}(X,\Id\mu) \to \mathsf{L}^{\beta}(X,\Id\mu)$ by 
$\left\|\cdot\right\|_{\mu,\alpha\to \beta}$. 
\newline
In the above situation, given a measurable function $F:X\to 
[0,\infty)$, we denote by $\Id\mu_F$ the measure $\Id \mu_F(x):=F(x)\Id\mu(x)$. \\The complex Hilbert space $\mathsf{L}^2(X,\Id\mu)$ is equipped with the 
scalar product 
$\langle f,h\rangle_{\mu}=\int_X \overline{f}h \Id \mu$, which is thus anti-linear 
in its first slot.\\
Let $\ILL$ denote space of of bounded operators between Banach spaces and 
for $q\in [1,\infty]$, let $\ILL^{q}$ denote the $q$-th Schatten class of 
bounded operators between Banach spaces. In particular, $q=2$ corresponds to 
the Hilbert-Schmidt case and $q=\infty$ to the compact case. We refer the 
reader to \cite{weid} for further notation (and the operator theoretic facts) that we will use in the sequel.

\section{An abstract result for operators on measure spaces}\label{sec1}

This section is devoted to the formulation and the proof of the above stated 
Theorem \ref{ddp}, as well as its generalization to measure theoretic vector 
bundles.

Let $(X,\mu)$ be a not necessarily complete sigma-finite 
measure space.  We will be interested in certain nonnegative operators 
(and perturbations thereof) in $\mathsf{L}^2(X,\Id\mu)$ which generate 
semigroups that are defined by appropriate integral kernels as follows. 

\begin{Definition}\label{heat} A measurable map
\begin{align}
p:(0,\infty)\times X\times X\longrightarrow [0,\infty),\>\>(t,x,y)\longmapsto  
p(t,x,y)\label{p1}
\end{align}
is called a {\it pointwise consistent $\mu$-heat-kernel}
if it satisfies the following properties: 
\begin{align}
%& p_t(x,\bullet)\in\mathrm{L}^2(X,\Id\mu),\nn\\
&\>\> p(t+s,x,y)=\int_X p(t,x,z)p(s,z,y) \Id\mu ( z)\>\>\text{ for all 
$t,s>0$, $x,y\in X$}\label{A1},\\
&\>\> p(t,x,y)=p(t,y,x) \>\>\text{ for all $t>0$, $x,y\in 
X$}\label{A2},\\
&\>\>\int_X p(t,x,y)\Id\mu( y)\leq 1\>\>\text{for all $t>0$, 
$x\in X$}\label{A3},\\
&\> \text{ with }\>
P_t f(x):=\int_X p(t,x,y) f(y)\Id\mu( y), \>\>t>0, f\in \mathsf{L}^2(X,\Id\mu),
\nn\\
&\>\>\text{ we have }\>\>  \lim_{t\to 0+} \left\|P_t f-f \right\|_{\mu,2}=0
\label{con}.
\end{align}
\end{Definition}
The above definition (see also \cite{lenz}) simply abstracts heat kernel properties that one hold on Riemannian manifolds. Two remarks are in order:

\begin{Remark} 1. Note that $P_t f$ is indeed a well-defined element 
of $\mathsf{L}^2(X,\Id\mu)$. To see this, note first that 
$p(t,x,\cdot)\in\mathsf{L}^2(X,\Id\mu)$ by (8), so that $x\mapsto P_t f(x)$ 
is a well-defined measurable function by Cauchy-Schwarz, and one can 
estimate with Cauchy-Schwarz and (\ref{A3}) as follows:
\begin{align}
&\left\|P_t f \right\|^2_{\mu,2} \leq \int_X\left(\int_X \sqrt{ p(t,x,y)} 
\sqrt{ p(t,x,y)} |f(y)|\Id\mu(y)\right)^2 \Id\mu( x)  \nn\\
&\leq \int_X  \int_X  p(t,x,z)\Id\mu( z)  \int_X p(t,x,y) |f(y)|^2 \Id\mu( y) 
\Id \mu( x) \nn\\
&\leq \int_X   \int_X p(t,x,y) \Id\mu( x)   |f(y)|^2 \Id\mu( y)   
\leq  \left\|f \right\|^2_{\mu,2},\label{contr}.   
\end{align}
\vspace{1.1mm}
%%%%%%%%%%%%%%%%%%%% Discuss
2. Definition \ref{heat} is in the spirit of  in \cite[Definition 2.1]{gri} 
(where the notion of \lq\lq{}$\mu$-heat-kernels\rq\rq{} is defined), with one 
essential difference: We have required (\ref{A1}), (\ref{A2}) and (\ref{A3}) to 
hold {\it pointwise} and {\it not} only in the $\mu_{\otimes_2}$ 
(the product measure) sense. Indeed, since one typically has 
$$
\mu_{\otimes_2}\big\{(x,y)\left|(x,y)\in X\times X, x=y\big\}\right.=0 
$$
in nondiscrete applications, we will need (\ref{A1}), (\ref{A2}) to hold 
pointwise in the proof of Theorem \ref{main} below. A slightly more general 
procedure to avoid this problem would have been to use Definition 2.1 in 
\cite{gri} literally and to formulate Assumption \ref{gaus} below in terms of the 
rhs of (\ref{A1}), with '$\mu$-$\mathrm{ess \ sup}$' instead of '$\sup$'. 
However, thinking of $p$ as the transition probability density of a Markoff 
process, we found the latter generalization unnecessary in view of the 
applications that we have in mind.
%%%%%%%%%%%%%%%%%%%%%%%
\end{Remark}

\emph{We fix an arbitrary pointwise consistent $\mu$-heat-kernel $p$ in the 
following.}\vspace{1mm}

In view of (\ref{contr}), setting $P_0:=\mathrm{id}_{\mathsf{L}^2(X,\Id\mu)}$, 
the family $(P_t)_{t\geq 0}\subset 
\ILL(\mathsf{L}^2(X,\Id\mu))$ defines a strongly continuous contraction 
semigroup of self-adjoint operators and we denote the generator of 
$(P_t)_{t\geq  0}$ with $H_p$, that is, $H_p$ is the unique self-adjoint 
nonnegative operator in $\mathsf{L}^2(X,\Id\mu)$ with 
$\mathrm{e}^{-t H_p}=P_t$ for all $t\geq 0$. 

Next, we are going to consider perturbations of $H_p$. We start with the 
following central definition.

\begin{Definition}\label{van} A measurable function $h:X\to\IC$ is said to 
{\it vanish $\mu$-weakly at $\infty$}, if 
$$
\mu\{|h|\geq c\}<\infty\>\> \text{ for all $c>0$}.
$$
\end{Definition}

This measure theoretic notion of \lq\lq{}vanishing at 
infinity\rq\rq{} appears in \cite{lieb} for the Lebesgue measure in $\IR^m$, 
in the context of rearrangement inequalities for relativistic kinetic energies 
(see also \cite{villa} for an analogous context). It is also used in \cite{lenzi}, where the authors examine the question of emptyness of the essential spectrum of operators such as $H_p$ and perturbations thereof.\\
Let us denote the 
$\mu$-equivalence classes of functions that vanish $\mu$-weakly at $\infty$ 
with $\mathsf{L}_{\infty}(X,\Id\mu)$ and define
\begin{equation}
 \mathsf{L}^{\infty}_{\infty}(X,\Id\mu):=\mathsf{L}^{\infty}(X,\Id\mu)
 \cap \mathsf{L}_{\infty}(X,\Id\mu). 
\end{equation}
The following can be said about the structures of these spaces.

\begin{Proposition} $\mathsf{L}_{\infty}(X,\Id\mu)$ is an algebra, and 
 $\mathsf{L}^{\infty}_{\infty}(X,\Id\mu)$ is a Banach algebra with respect to 
 $\left\|\cdot\right\|_{\mu,\infty}$.
\end{Proposition}
\begin{proof} The only nonobvious assertion is that 
    $\mathsf{L}^{\infty}_{\infty}(X,\Id\mu)$ is a closed subspace of 
    $\mathsf{L}^{\infty}(X,\Id\mu)$. To see this, assume that 
    $(h_n)\subset \mathsf{L}^{\infty}_{\infty}(X,\Id\mu)$ is a sequence with 
    $h_n\to h$ as $n\to\infty$ in $\mathsf{L}^{\infty}(X,\Id\mu)$. Then there 
    is a set $N\subset X$ of measure 0 such that 
    $(h_{n})$ converges uniformly in $X-N$. Hence we can find for 
    $c>0$ a number $n(c)$ such that $|(h - h_{n})(x)| \le c/2$, for 
    $n\ge n(c)$ and
    all $x\in h^{-1}(x\ge c)\cap (X-N)$. Then for such $x$ we have
    $c/2\le |h_{n}(x)|$ and thus $h\in\mathsf{L}^{\infty}_{\infty}(X,\Id\mu)$,
    by the finiteness of $\mu(N)$.
\end{proof}

For any measurable $V:X\to\IC$, the corresponding (maximally defined) 
multiplication operator in $\mathsf{L}^2(X,\Id\mu)$ will be denoted by
$\hat{V}$ in the following, that is,
\[
\mathrm{Dom}(\hat{V})=\left.\big\{f\right|f\in\mathsf{L}^{2}(X,\Id\mu), 
Vf\in \mathsf{L}^{2}(X, \Id\mu)\big\},\>\>\hat{V}f(x):=V(x)f(x). 
\]
Clearly, $\hat{V}$ is always a normal operator in $\mathsf{L}^{2}(X,\mu)$ 
which only depends on the $\mu$-equivalence class of $V$, and which is 
self-adjoint if and only if $V$ is ($\mu$-a.e.) real-valued. 

\vspace{1.2mm}

\begin{Definition}\label{gaus} Given $q\geq 1$, a function $F_1 :X\to (0,\infty)$ is called an \emph{$\mathsf{L}^q$-control function for $p$}, if there exists a measurable function $F_2 :(0,\infty)\to (0,\infty)$ such that
\begin{align}
&p(t,x,x)\leq F_1(x)F_2(t)\text{ for all $t>0$, $x\in 
X$, and }\label{abdd}\\
&\int^{\infty}_{0}\mathrm{e}^{-t}\,F_2(t)^{1/(2q)}\,\Id t<\infty.\label{ak}
\end{align}
\end{Definition} 

\begin{Remark}
It is readily seen that the function $\tilde{F}_2(t) := \mathrm{const.}F_2(t)$ satisfies the same assumption as $F_2$. Hence for bounded $F_1$ one can assume that $F_1 \equiv 1$ without loss of generality. Furthermore, we will prove later that appropriate control functions exist on weighted Riemannian manifolds (cf. Theorem \ref{gri}) and weighted graphs, without further assumptions on the geometry (cf. inequality (\ref{dzzu})).
\end{Remark}

Now we can prove the following abstract result.        

\begin{Theorem} \label{main} Given $q\geq 1$ and an an $\mathsf{L}^q$-control function $F_1$ for $p$, assume that 
 $V:X\to\IR$ admits a decomposition $V=V_1+V_2$ such that 
 either  
\begin{align}\label{cond}
 q=1,\,\,V_1\in\mathsf{L}^{2}(X,F_1 \Id\mu),\,\,\,
 V_2\in\mathsf{L}_{\infty}(X,F_1\Id\mu)\cap \mathsf{L}^{\infty}_{\infty}(X,\Id\mu),
\end{align}
 or 
\begin{align}\label{cond2}
 q>1,\,\,\, \text{$F_1 \equiv 1$},\,\,\,V_1\in\mathsf{L}^{2q}(X,\Id\mu),\,\,\,
 V_2\in\mathsf{L}^{\infty}_{\infty}(X,\Id\mu).
\end{align}
Then for all $a>0$, $\hat{V}(H_p+a)^{-1}\in\ILL^{\infty}
(\mathsf{L}^2(X,\Id\mu))$. In particular, 
$\mathrm{Dom}(H_p+\hat{V})=\mathrm{Dom}(H_p)$ and 
$H_p+\hat{V}$ is self-adjoint and semibounded from below with 
$\sigma_{\mathrm{ess}}(H_p+\hat{V})= \sigma_{\mathrm{ess}}(H_p)$.
\end{Theorem}

\begin{Remark} The compactness of $\hat{V}(H_p+a)^{-1}$ also 
implies that $H_p$ and $H_p+\hat{V}$ have the same singular sequences and 
that any operator core for $H_p$ is also an operator core for 
$H_p+\hat{V}$.
\end{Remark}

For the proof of Theorem \ref{main}, we will need the following 
estimate. 

\begin{Lemma}\label{ddo} Assume that there exist $t>0$ and a 
measurable set $U\subset X$ such that 
$$
C_U(t):=\sup_{x\in U,y\in X}p(t,x,y)<\infty.
$$
Then for every $\alpha > 2$ one has
$$
\left\|\widehat{1_{U}} P_t\right\|_{\mu,2\to\alpha}
\leq C_{U}(t)^{\f{\alpha-2}{2\alpha}}<\infty.
$$
\end{Lemma}

\begin{proof} 
Let $f\in\IL^2(X,\Id\mu)$, then from (\ref{A3}) we see that
\begin{equation}
\|f\|_{\mu,2}^{2}\ge\int_{U\times X}p(t,x,y)|f(y)|^2 {\rm d}\mu(x) 
{\rm d}\mu(x),
\end{equation}
which suggests the following application of Hölder's inequality with exponents 
\[
p_1=\alpha,\>\>p_2 = 2,\>\>p_3=\f{2\alpha}{\alpha-2}.
\]
We estimate 
\begin{align}
&\left\|\widehat{1_{U}}P_t f\right\|^{\alpha}_{\mu,\alpha}
\leq 
\int_U \left(\int_X p(t,x,y)|f(y)| \Id \mu(y)\right)^{\alpha} \Id \mu(x)\\
&= \int_U\left( \int_X \left(  p(t,x,y)|f(y)|^{2}\right)^{\f{1}
{\alpha}}p(t,x,y)^{1-\f{1}{\alpha}}|f(y)|^{1-\f{2}{\alpha}}    
\Id\mu(y)\right)^{\alpha} \Id\mu(x)\nn\\
&\leq \int_U \left(\int_X p(t,x,y) |f(y)|^{2} \Id\mu(y) \right) 
\left(\int_X p(t,x,y)^{2(1-\f{1}{\alpha})}\Id\mu(y)  
\right)^{\alpha/2}\nn\\
&\>\>\>\>\times 
\left( \int_X |f(y)|^{2}\Id\mu(y) \right)^{\f{1}{2}
\left({\alpha}-2 \right) } 
\Id\mu(x),\nn
\end{align}
so that we get from (18) and the definition 
of $C_U(t)$ 
\begin{align*}
\left\|\widehat{1_{U}} P_tf\right\|^{\alpha}_{\mu,\alpha}\leq  \left\|f\right\|^{2 }_{\mu,2} 
C_U(t)^{\f{1}{2}\left(\alpha-2 \right)}
\left\|f\right\|^{\alpha-2}_{\mu,2} =    C_U(t)^{\f{1}{2}\left( \alpha-2\right)}
\left\|f\right\|^{\alpha }_{\mu,2},
\end{align*}
which completes the proof.
\end{proof}
%%%%%%%%%%%%%%%%%%%%%%%%%%%%%%%%%%%%%%%%%%%%%
%%%%%%%%%%%%%%%%%%%%%%%%%%%%%%%%%%%%%%%%%%%%%
\begin{proof}[Proof of Theorem \ref{main}] Pick a function $F_2$ as in Definition \ref{gaus}. We are going to prove 
$$
\hat{V}(H_p+1)^{-1}\in\ILL^{\infty}(\mathsf{L}^2(X,\Id\mu)), 
$$
so that all assertions will follow from well-known abstract perturbation theory 
results on linear operators (see for example Theorem 9.1 and Theorem 9.9 in 
\cite{weid}). We will freely use the fact that for every operator $A$ 
in $\mathsf{L}^2(X,\Id\mu)$ and every $r\in [1,\infty]$, one has
$$
A(H_p+1)^{-1}\in\ILL^{r}(\mathsf{L}^2(X,\Id\mu))\Leftrightarrow 
A(H_p+a)^{-1}\in\ILL^{r}(\mathsf{L}^2(X,\Id\mu))\text{ for all $a>0$},
$$
which follows from the first resolvent identity.\vspace{1mm}

{\it Step 1: If $q=1$, then one has $\hat{W}(H_p+a)^{-1}\in 
\ILL^2(\mathsf{L}^2(X,\Id\mu))$ for all $a\geq 2$, and all 
$W\in \IL^2(X,F_1\Id\mu)$.}\vspace{2mm}

Proof: For all $t>0$, with (\ref{A1}) and (\ref{A3}) we can estimate
the Hilbert-Schmidt norm of $\hat{W}P_t$ as follows:
\begin{align*}
&\left\|\hat{W}P_t\right\|_{\mu, \mathrm{HS}}^2=\int_X|W(x)|^2
\int_X  p(t,x,y)^2 \Id\mu(y)\Id\mu(x)\\
&=\int_X|W(x)|^2  p(2t,x,x)\Id\mu(x)\\
&\leq \int_X|W(x)|^2 F_1(x)F_2(2t)\Id\mu(x) \\
&= F_2(2t)\left\|W\right\|_{\mu_{F_1}, 2}^2.
\end{align*}
Taking the Laplace transform, we have for $a\geq 2$ 
\begin{align*}
&\left\|\hat{W}(H_p+a)^{-1}\right\|_{\mu, \mathrm{HS}}= 
\left\|\hat{W}\int^{\infty}_0\mathrm{e}^{-at}P_t\Id t\right\|_{\mu, \mathrm{HS}}\\
&\leq\left\|W\right\|_{\mu_{F_1}, 2}
\int^{\infty}_0 F_2(2t)^{1/2}\mathrm{e}^{-at}\Id t<\infty.
\end{align*}

{\it Step 2: If $q>1$ and $F_1\equiv 1$, then for any 
$W\in\mathsf{L}^{2q}(X,\Id\mu)$, $t>0$, it holds that 
$\hat{W}P_t\in \ILL(\mathsf{L}^2(X,\Id\mu))$ with}
$$
\left\|\hat{W}P_t\right\|_{\mu, 2\to 2}\leq  F_2(t)^{\f{1}{2q}}
\left\|W\right\|_{\mu, 2q}<\infty.
$$

Proof: Using (\ref{A1}), (\ref{A2}), and Cauchy-Schwarz we get
$$
C_X(t):=\sup_{x,y\in X}\int p(t/2,x,z)p(t/2,z,y)\Id\mu(z)
\leq \sup_{x\in X}p(t,x,x)\leq  F_2(t)<\infty.
$$

Now let $f \in \IL^2(X,\Id\mu)$.  We use Hölder\rq{}s inequality with 
$1/q+1/q^*=1$ and Lemma \ref{ddo} with $U=X$ and 
$\alpha= 2q^{*}>2$ to get 
\begin{align*}
&\left\|\hat{W}P_tf\right\|_{\mu, 2}^2= 
\int_{X} |W(x)|^2  |P_t f(x)|^2\Id\mu(x)\\
&\leq \left(\int_{X} |W(x)|^{2q}\Id\mu(x)\right)^{\f{1}{q}}
\left(\int_{X} |P_t f(x)|^{2q^*}\Id\mu(x)\right)^{\f{1}{q^*}}\\
&\leq \left\|W\right\|_{\mu, 2q}^2 
\left\|P_t f\right\| ^{2}_{\mu, 2q^*}
\leq  \left\|W\right\|_{\mu, 2q}^2 
\left\|P_t \right\| ^{2}_{\mu, 2\to 2q^*}
\left\|f\right\| ^{2}_{\mu, 2}\leq C_X(t)^{\f{1}{q}}
\left\|W\right\|_{\mu, 2q}^2 \left\|f\right\| ^{2}_{\mu, 2}\\
&\leq  F_2(t)^{\f{1}{q}}\left\|W\right\|_{\mu, 2q}^2 
\left\|f\right\| ^{2}_{\mu, 2}.
\end{align*}

\vspace{2mm}

{\it Step 3: If $q>1$ and $F_1\equiv 1$, then for any 
$W\in\mathsf{L}^{2q}(X,\Id\mu)$, and all $a\geq 1$ it holds that 
$\hat{W}(H_p+a)^{-1}\in \ILL(\mathsf{L}^2(X,\Id\mu))$ with}
$$
\left\|\hat{W}(H_p+a)^{-1}\right\|_{\mu, 2\to 2}\leq  
\left\|W\right\|_{\mu, 2q}\int^{\infty}_0F_2(t)^{\f{1}{2q}} 
\mathrm{e}^{-at}\Id t<\infty.
$$

Proof: From the previous step we get, taking the Laplace transform,
\begin{align*}
&\left\|\hat{W}(H_p+a)^{-1}\right\|_{\mu, 2\to 2}= 
\left\|\hat{W}\int^{\infty}_0\mathrm{e}^{-at}P_t
\Id t\right\|_{\mu, 2\to 2}
\leq\int^{\infty}_0 \left\|\hat{W}P_t\right\|_{\mu, 2\to 2}
\mathrm{e}^{-at}\Id t\\
& \leq   \left\|W\right\|_{\mu, 2q}\int^{\infty}_0F_2(t)^{\f{1}{2q}} 
 \mathrm{e}^{-at}\Id t.
\end{align*}

\emph{Step 4: Let $F_1\equiv 1$, let $W:X\to\IR$ be measurable, 
and set $W_n:=1_{X_n}\min(n,W)$, $n\in\IN$, where $(X_n)$ is an 
exhaustion of $X$ with $\mu(X_n)<\infty$ (remember that $(X,\mu)$ 
is sigma-finite). Then for all $n$ one has 
$\hat{W_n}(H_p+1)^{-1}\in\ILL^{\infty}(\mathsf{L}^{2}(X,\Id\mu))$.}\vspace{2mm}

Proof: Indeed, as $\hat{W_n}$ is bounded, one has the equivalence (cf. \cite{lenzi}, Theorem 1.3)
$$
\hat{W_n}(H_p+1)^{-1}\in \ILL^{\infty}(\mathsf{L}^{2}(X,\Id\mu)) \Leftrightarrow \hat{W_n}P_t\in \ILL^{\infty}(\mathsf{L}^{2}(X,\Id\mu)), 
$$
for all $t>0$. But for any $t>0, \hat{W_n}P_t$ is Hilbert-Schmidt:
\begin{align*}
&\int_{X\times X}|W_n(x)|^2p(t,x,y)^2 
\Id\mu (x)\Id\mu (y)\\
&=\int_X |W_n(x)|^2p(2t,x,x) \Id\mu(x)=\int_{X_n} |W_n(x)|^2p(2t,x,x) \Id\mu(x)\\
&\leq n^2 \mu(X_n) \sup_{x\in X}p(2t,x,x)
\leq n^2 \mu(X_n) F_2(2t)<\infty,
\end{align*}
where we have used (\ref{A1}) once more. \vspace{2mm}

{\it Step 5: For any\footnote{Note here that in case $F_1\equiv 1$, and $q=1$, the conditions (\ref{cond2}) and (\ref{cond}) are equivalent, so that we can treat both cases $q=1$ and $q>1$ simultaneously from here on.} $W\in\mathsf{L}^{2q}(X,F_1\Id\mu)$ it holds that $\hat{W}(H_p+1)^{-1}\in \ILL^{\infty}(\mathsf{L}^{2}(X,\Id\mu))$.}\vspace{2mm}

Proof: In case $q=1$ this follows from step 1. In case $q>1$ and $F_1\equiv 1$, we use that limits in operator norm of compact operators are compact as well, this follows from Step 3 and Step 4: As $n\to\infty$, we get from these results and dominated convergence that 
\begin{align*}
 \left\|\hat{W}(H_p+1)^{-1}-\hat{W_n}(H_p+1)^{-1} \right\|_{\mu,2\to 2}\leq C\left\|W-W_n\right\|_{\mu, 2q}\to 0,
\end{align*}
as claimed.\vspace{2mm}

{\it Step 6: With $Y_n:=\{|V_2|\geq 1/n\}$ one has } 
\[
(\widehat{V_1}+\widehat{1_{Y_n} V_2})(H_p+1)^{-1}\in \ILL^{\infty}(\mathsf{L}^2(X,F_1\Id\mu))\>\>\text{ for any $n\in\IN$}.
\]

Proof: Clearly $V_2\in \mathsf{L}^{\infty}(X,\Id\mu)\cap \mathsf{L}_{\infty}(X,F_1\Id\mu)$ implies 
$$
V_2\in \mathsf{L}^{\infty}_{\infty}(X,F_1\Id\mu),
$$
so that we have
$$
1_{Y_n} V_2 \in \mathsf{L}^{2q}(X,F_1\Id\mu)\text{,  $V_1+1_{Y_n} V_2\in \mathsf{L}^{2q}(X,F_1\Id\mu)$} 
$$
and the assertion follows from step 3. \vspace{1.2mm}

{\it Step 7: One has $\widehat{V}(H_p+1)^{-1}\in\ILL^{\infty}(\mathsf{L}^2(X,\Id\mu))$.}

Proof: Clearly, $\widehat{V}(H_p+1)^{-1}$ is in $\ILL(\mathsf{L}^2(X,\Id\mu))$ by the previous considerations. Using that the limits in operator norm of compact operators are compact, it is sufficient to prove 
\[
 \left\|\widehat{V}(H_p+1)^{-1}-(\widehat{V_1}+\widehat{1_{Y_n} V_2})(H_p+1)^{-1} \right\|_{\mu,2\to 2}\to 0\>\>\text{as $n\to\infty$,}
\]
which follows, if we can show that $\|\widehat{V_2}-\widehat{1_{Y_n} V_2} \|_{\mu,2\to 2}\to 0$ as $n\to\infty$. To see the latter convergence, just note that for any $f\in \mathsf{L}^2(X,\Id\mu)$ one has 
\begin{align}
& \| (\widehat{V_2}-\widehat{1_{Y_n} V_2}) f \|_{\mu,2}\leq \|V_2-1_{Y_n} V_2\|_{\mu,\infty} \left\|f\right\|_{\mu,2},\>\text{ so that}\nn\\
&\| \widehat{V_2}-\widehat{1_{Y_n} V_2} \|_{\mu,2}\leq \|V_2-1_{Y_n} V_2\|_{\mu,\infty},\nn
\end{align}
and this tends to $0$ as $n\to\infty$, which follows readily from $V_2\in\mathsf{L}_{\infty}(X,\Id\mu)$. This completes the proof. 
\end{proof}

In the rest of this section, we are going to extend the above result to a class of operators that act on sections in certain finite dimensional vector bundles over $X$ (this is the content of Corollary \ref{mag1} below). 

\begin{Definition}\label{vb} 1. Let $E$ and $Y$ be measurable spaces. A surjective measurable map $\pi:E\to Y$ is called a \emph{(complex) measurable vector bundle over $Y$ with rank $d\in\IN$}, if for any $y\in Y$, the set $\pi^{-1}\{y\}$ is a complex linear space and there is a measurable set $U\subset Y$ with $y\in U$ and a measurable bijection 
$$
\phi:\pi^{-1}(U)\longrightarrow U \times \IC^d\>\text{ with $\phi^{-1}$ also measurable,}
$$
called a \emph{vector bundle chart with domain $U$} for $\pi: E\to Y$, with the following properties: The diagram
\begin{align}
\xymatrix{
\pi^{-1}(U) \ar[d]_{\pi}   \ar[r]^{\phi}  &  U\times \IC^d \ar[dl]^{\mathrm{pr}_1}\\
Y    
}\label{dia}
\end{align} 
commutes and $\phi^{-1}(y,\cdot):\IC^d\to \pi^{-1}\{y\}$ is an isomorphism of complex linear spaces.

2. A measurable vector bundle $\pi: E\to Y$ is called \emph{countably generated}, if the there is an $\IN$-indexed cover $Y=\bigcup_{n\in\IN} U_n$ with measurable sets $U_n\subset Y$ such that each $U_n$ is the domain of a vector bundle chart for $\pi: E\to Y$. 
\end{Definition}

\begin{Remark} 1. As usual, for any measurable space $Y$ one gets the canonical measurable vector bundles $\mathrm{pr_1}:Y\times\IC^d\to Y$ with rank $d$, in particular, $d=1$ corresponds to the scalar situation of Theorem \ref{main}.

2. In typical applications, the space $Y$ has some additional structure such that in fact all measurable vector bundles of interest over $Y$ are countably generated: For example, if $Y$ is a topological space which is Lindelöf (e.g., $Y$ could be metrizable and second countable), then any continuous vector bundle $\pi: E\to Y$ is a measurable one in the sense of Definition \ref{vb} (if $E$ and $Y$ are equipped with their Borel-sigma-algebras), which is automatically countably generated.
\end{Remark}

Let $\pi: E\to Y$ be a measurable vector bundle with rank $d$ in the sequel. Writing $E_y:=\pi^{-1}\{y\}$ for the fibers of $\pi$ will cause no danger of confusion, where then accordingly we can and will simply write $E\to Y$ instead of $\pi: E\to Y$. \\
Analogously to the smooth case, using any sufficiently regular multifunctor\footnote{To be precise, we mean multifunctors that are Borel measurable with respect to the canonical topology on finite dimensional linear spaces} of complex linear spaces, we can construct new measurable vector bundles from old ones, for example, the set of endomorphisms on $E$ given by
\begin{align}
\mathrm{End}(E):=\bigsqcup_{y\in Y}\mathrm{End}(E_y)\longrightarrow Y\label{dsaa}
\end{align}
becomes a measurable vector bundle with rank $d^2$ in a canonical way: Any vector bundle chart $\bigsqcup_{y\in U}E_y\to U \times \IC^d$ induces an obvious map
\begin{align}
\bigsqcup_{y\in U}\mathrm{End}(E_y)\longrightarrow  U \times \mathrm{End}(\IC^{d})=U \times \IC^{d^2},\label{asa}
\end{align}
and as $U$ runs through a measurable cover of $Y$, the collection (\ref{asa}) determines a sigma-algebra on $\mathrm{End}(E)$, and, a posteriori, (\ref{dsaa}) indeed becomes a measurable vector bundle with rank $d\times d$, where the vector bundle charts are then given by the maps (\ref{asa}).\vspace{1.2mm}

Given a set $U\subset Y$, we will denote the measurable sections in $E\to Y$ over $U$ with $\Gamma(U,E)$, that is, $\Gamma(U,E)$ is the set of measurable maps $\psi:U\to E$ that satisfy $\pi\circ\psi =\mathrm{1}_U$.

\begin{Definition}\label{frame} 1. Let $U\subset Y$. Then a collection $e_1,\dots,e_d\in \Gamma(U,E)$ is called a \emph{measurable frame} for $E\to Y$ over $U$, if the vectors $e_1(y),\dots,e_d(y)$ form a basis of $E_y$ for all $y\in U$.  

2. A \emph{measurable Hermitian structure} for $E\to Y$ is given by a collection of complex scalar products $(\cdot,\cdot)_y:E_y\times E_y\to \IC$, $y\in Y$, with the property that for any fixed $\psi_1,\psi_2\in \Gamma_{}(Y,E)$, the function $y\mapsto (\psi_1(y),\psi_2(y))_y$ is measurable. Then the pair given by $E\to Y$ and $(\cdot,\cdot)_{\cdot}$ is called a \emph{measurable Hermitian vector bundle}, and for $U\subset Y$, a collection $e_1,\dots,e_d\in \Gamma_{}(U,E)$ is called a \emph{measurable orthonormal frame} for the latter pair over $U$, if one has 
$$
(e_i(y),e_j(y))_y=\begin{cases}1, &\text{if $i=j$}\\0, &\text{else}\end{cases}\ \text{ for all $y\in U$ and $i,j=1,\dots, d$.} 
$$
\end{Definition}

In what follows, there will be no danger of confusion in ommitting the dependence of measurable Hermitian structures on their underlying vector bundles (as we have done in Definition \ref{frame}.2). Given such a structure $(\cdot,\cdot)_{\cdot}$ on $E\to Y$, we denote with $|\cdot|_y:=\sqrt{(\cdot,\cdot)_y}$ the corresponding norm on $E_y$, and also the operator norm on $\mathrm{End}(E_y)$. Then we can safely continue to denote such a measurable Hermitian vector bundle with $E\to Y$. \vspace{1.2mm}

Let us now explain how one can use the above considerations in order to define bundle-valued $\mathsf{L}^{q}$-type Banach spaces with respect $(X,\Id\mu)$ for all $q\in [1,\infty]$: To this end, if now $E\to X$ is a measurable Hermitian vector bundle with rank $d$, and if $q\in [1,\infty)$, then $\Gamma_{\mathsf{L}^{q}}(X,E;\Id\mu)$ denotes the complex normed space of $\mu$-equivalence classes of measurable sections $f$ in $E\to X$ over $X$ with norm
$$
\left\|f\right\|_{\mu,q}=\left(\int_X |f(x)|^{q}_x \Id\mu( x)\right)^{\f{1}{q}}<\infty,
$$ 
and the norm $\left\|f\right\|_{\mu,\infty}$ on the space $\Gamma_{\mathsf{L}^{\infty}}(X,E;\Id\mu)$ is defined as the infimum of all $C>0$ such that $|f(x)|_x\leq C$ for $\mu$-a.e. $x\in X$. Using the same proof as for functions, one sees that all the $\Gamma_{\mathsf{L}^{q}}(X,E;\Id\mu)$\rq{}s are in fact complex Banach spaces, in particular $\Gamma_{\mathsf{L}^{2}}(X,E;\Id\mu)$ becomes a complex Hilbert space in the obvious way. 

The following result generalizes a result by D. Pitt \cite{pitt} to vector bundles.

\begin{Theorem} \label{pits} Let $E\to X$ be a measurable Hermitian countably generated vector bundle with rank $d$, let $(X\rq{}, \mu\rq{})$ be another sigma-finite measure space with $E\rq{}\to X\rq{}$ a measurable Hermitian countably generated vector bundle with rank $d\rq{}$. Assume that for some $q\rq{}\in (1,\infty)$, $q\in [1,\infty]$ we are given operators 
\[
A\in\ILL\big(\mathsf{L}^{q\rq{}}(X\rq{},\Id\mu\rq{}),\mathsf{L}^{q}(X,\Id\mu)\big), \>B\in\ILL\big(\Gamma_{\mathsf{L}^{q\rq{}}}(X\rq{},E\rq{};\Id\mu\rq{}),\Gamma_{\mathsf{L}^{q}}(X,E;\Id\mu)\big) 
\]
with the property that $A$ is positivity preserving and that for arbitrary $f\in\Gamma_{\mathsf{L}^{q\rq{}}}(X\rq{},E\rq{};\Id\mu\rq{})$ one has $|Bf(x)|_x\leq A|f|(x)$ for $\mu$-a.e. $x\in X$. Then the implication
\begin{align}
A\in\ILL^{\infty}\big(\mathsf{L}^{q\rq{}}(X\rq{},\Id\mu\rq{}),\mathsf{L}^{q}(X,\Id\mu)\big) \Rightarrow B\in\ILL^{\infty}\big(\Gamma_{\mathsf{L}^{q\rq{}}}(X\rq{},E\rq{};\Id\mu\rq{}),\Gamma_{\mathsf{L}^{q}}(X,E;\Id\mu)\big)\nn
\end{align}
holds true.
\end{Theorem}

\begin{proof} We are going to reduce the bundle-valued problem by a measure theoretic localization argument to the scalar situation in \cite{pitt}. Assume we are given a cover  
$$
X=\bigcup_{n\in\IN} U_n,\>\text{ with vector bundle charts }\>\phi^{(n)}:\bigcup_{x\in U_n} E_{x}\longrightarrow U_n\times \IC^d. 
$$
Then for each $n$ we get a frame $e^{(n)}_1,\dots,e^{(n)}_{d}\in \Gamma(U_n,E)$ by setting $e^{(n)}_j(x):=\phi^{(n),-1}(x,v_j)$, $x\in U_n$, where $v_1,\dots,v_d\in\IC^d$ is some basis. Let us define measurable sets $V_n\subset X$ by
$$
V_1:=U_1, \>V_n:=U_n-  \Big(\bigcup_{l\in\IN:\> l\ne n } U_l\Big)\>\text{ for $n\geq 2$, so that $X=\bigsqcup_{n=1}^\infty V_n$}. 
$$
Then we can define \emph{globally defined measurable} frames $e_1,\dots,e_d\in \Gamma(X,E)$ by setting $e_j\mid_{V_n}:=e^{(n)}_j\mid_{V_n}$ for $j=1,\dots,d$, and, in view of the Gram-Schmidt algorithm, we can and will assume this frame to be orthonormal, and then this frame induces the isometric isomorphism 
\begin{align}
&\Psi:\Gamma_{\mathsf{L}^{q}}(X,E;\mu)\longrightarrow \mathsf{L}^{q}\left(X,\Id\mu;\IC^{d}\right), \nn\\
&\Psi f(x):=\big[(f(x),e_1(x))_x,\dots,(f(x),e_d(x))_x\big].\nn
\end{align}

Likewise, we can also construct a orthonormal frame $e_1\rq{},\dots,e_{d\rq{}}\rq{}\in \Gamma(X\rq{},E\rq{})$ and the corresponding  the isometric isomorphism
\begin{align*}
&\Psi\rq{}:\Gamma_{\mathsf{L}^{q\rq{}}}(X\rq{},E\rq{};\Id\mu\rq{})\longrightarrow \mathsf{L}^{q\rq{}}(X\rq{},\Id\mu\rq{};\IC^{d\rq{}}),\nn\\ 
&\Psi\rq{}f(x):=\big[(f(x),e_1\rq{}(x))_x,\dots,(f(x),e_d\rq{}(x))_x\big].
\end{align*}

Let us define a bounded linear operator $\tilde{B}$ by the diagram
\begin{align}
\begin{xy}
\xymatrix{
\Gamma_{\mathsf{L}^{q\rq{}}}(X\rq{},E\rq{};\Id\mu\rq{})   \ar[rr]^{B} & & \Gamma_{\mathsf{L}^{q}}(X,E;\Id\mu)\ar[dd]^{\Psi}\\\\
\mathsf{L}^{q\rq{}}(X\rq{},\Id\mu\rq{};\IC^{d\rq{}})\ar[uu]^{(\Psi\rq{})^{-1}}\ar[rr]_{\tilde{B}}  & & \mathsf{L}^{q}(X,\Id\mu;\IC^{d})
}\end{xy}\nn
\end{align}
Then clearly the proof is complete, if we can show that $\tilde{B}$ is compact. To this end, note that $\tilde{B}$ operates as follows:
\begin{align*}
\tilde{B}[f_1,\dots,f_{d\rq{}}](x)=\sum^{d\rq{}}_{j=1}\left[\big(B(f_je_j\rq{})(x),e_1(x)\big)_x,\dots,\big(B(f_je_j\rq{})(x),e_d(x)\big)_x\right].\nn
\end{align*}
If for each $j=1,\dots,d\rq{}$ and $i=1,\dots,d$ we define the bounded linear operator 
$$
\tilde{B}_{ij}:\mathsf{L}^{q\rq{}}(X\rq{},\Id\mu\rq{})\longrightarrow \mathsf{L}^{q}(X,\Id\mu),\>\tilde{B}_{ij}f(x):=\big(B(fe_j\rq{})(x),e_i(x)\big)_x,
$$
then the inequality $|\tilde{B}_{ij}f(x)|\leq A|f|(x)$, valid for $\mu$-a.e. $x\in X$, shows that we can use the main result from \cite{pitt} (in combination with Remark I therein, where the seperability assumptions are removed a posteriori) to deduce that $\tilde{B}_{ij}$ is compact. Finally, with 
$$
P_j\rq{}: \mathsf{L}^{q\rq{}}(X\rq{},\Id\mu\rq{};\IC^{d\rq{}})\longrightarrow \mathsf{L}^{q\rq{}}\left(X\rq{},\Id\mu\rq{}\right),\>[f_1,\dots,f_{d\rq{}}]\longmapsto f_j
$$ 
the canonical projection, one has $\tilde{B}=\bigoplus^d_{i=1} \sum^{d\rq{}}_{j=1}\tilde{B}_{ij}P_j\rq{}$, so that the latter operator indeed is compact, and the proof is complete.
\end{proof}

\begin{Remark} We have shown in the proof of Lemma \ref{pits} that any measurable countably generated vector bundle admits a globally defined measurable frame. However, it should be noted that in typical applications one starts with structures that are much \lq\lq{}smoother\rq\rq{} than measurable, and then the latter measurable identification, though global, is not very useful. This is the case, for example, in Section \ref{Riemann}, where we consider linear partial differential operators with smooth coefficients. 
\end{Remark}

\emph{For the rest of this section, we fix a measurable Hermitian vector bundle $E\to X$ with rank $d$.} \vspace{1.4mm}

We will need the following simple observation, which follows immediately from taking Laplace and Post-Wedder transforms (see for example \cite{hess}):

\begin{Lemma}\label{dad} Let $S$ be a self-adjoint nonnegative operator in $\mathsf{L}^2(X,\Id\mu)$, and let $T$ be a self-adjoint nonnegative operator in $\Gamma_{\mathsf{L}^{2}}(X,E;\Id\mu)$. Then the following two statements are equivalent: 
\begin{align}
&\emph{i)}\>\text{ For any $t\geq 0$, $f\in\Gamma_{\mathsf{L}^{2}}(X,E;\Id\mu)$ one has}\nn\\
&\>\>\>\>\>\>\text{  $\left|\mathrm{e}^{-t T}f(x)\right|_x\leq  \mathrm{e}^{-t S}|f|(x)$ for $\mu$-a.e. $x\in X$,}  \\
&\emph{ii)}\>\text{ For any $a>0$, $f\in\Gamma_{\mathsf{L}^{2}}(X,E;\Id\mu)$ one has} \nn\\
&\>\>\>\>\>\>\text{ $\left|(T+ a)^{-1}f(x)\right|_x\leq  (S+a)^{-1}|f|(x)$ for $\mu$-a.e. $x\in X$.}
\label{domi4} 
\end{align}
Given i) or ii), one necessarily has $\min\sigma(T)\geq \min\sigma(S)$.
\end{Lemma}

\begin{Definition} In the situation of Lemma \ref{dad}, if one of the equivalent conditions i), ii) is satisfied, we will write $T\succeq S$ and say that \emph{$T$ dominates $S$ in the Kato sense}.
\end{Definition}

Generalizing the scalar case, one has the following canonical notion of multiplication operators on vector bundles: For any $W\in \Gamma(X,\mathrm{End}(E))$, the corresponding \emph{maximally defined multiplication operator} in $\Gamma_{\mathsf{L}^{2}}(X,E;\Id\mu)$ is given as follows:
\begin{align*}
&\mathrm{Dom}(\hat{W})=\left.\big\{f\right|f\in\Gamma_{\mathsf{L}^{2}}(X,E;\Id\mu), Wf\in \Gamma_{\mathsf{L}^{2}}(X,E;\Id\mu)\big\},\>\hat{W}f(x)=W(x)f(x). 
\end{align*}
The operator $\hat{W}$ is depends only on the $\mu$-equivalence class of $W$, and is self-adjoint if and only if $W$ is ($\mu$-a.e.) pointwise self-adjoint, which can be seen with the same arguments as in the scalar situation. 

Being prepared with these observations, we can now prove the following generalization of Theorem \ref{main} to vector bundles:

\begin{Corollary} \label{mag1} Given $q\geq 1$ and an $\mathsf{L}^{q}$-control function for $p$, let $E\to X$ is countably generated, and let $\tilde{H}$ be a self-adjoint nonnegative operator in $\Gamma_{\mathsf{L}^2}(X,E;\Id\mu)$ which satisfies the Kato domination property $\tilde{H}\succeq H_p$. Furthermore, assume that $W\in\Gamma(X,\mathrm{End}(E))$ admits a decomposition $W=W_1+W_2$, with $W_j\in \Gamma(X,\mathrm{End}(E))$ pointwise self-adjoint and 

\begin{enumerate}
\item[$\cdot$] either $q=1$ and
$$
|W_1|\in\mathsf{L}^{2 }(X,F_1\Id\mu),\>|W_2|\in\mathsf{L}^{\infty}_{\infty}(X,\Id\mu)\cap \mathsf{L}_{\infty}(X,F_1\Id\mu).  
$$
\item[$\cdot$] or $q>1$, $F_1\equiv 1$ and 
$$
|W_1|\in\mathsf{L}^{2q}(X,\Id\mu),\>|W_2|\in\mathsf{L}^{\infty}_{\infty}(X,\Id\mu).  
$$

\end{enumerate}
Then one has 
$$
\hat{W}(\tilde{H}+a)^{-1}\in\ILL^{\infty}(\Gamma_{\mathsf{L}^2}(X,E;\Id\mu))\>\text{ for all $a>0$}.
$$
 In particular, $\mathrm{Dom}(\tilde{H}+\hat{W})=\mathrm{Dom}(\tilde{H})$, and $\tilde{H}+\hat{W}$ is self-adjoint and semibounded from below with $\sigma_{\mathrm{ess}}(\tilde{H}+\hat{W})= \sigma_{\mathrm{ess}}(\tilde{H})$.
\end{Corollary}

\begin{proof} By Lemma \ref{dad}, for any $a>0$, $f\in\Gamma_{\mathsf{L}^{2}}(X,E;\Id\mu)$ one has
\begin{align}
\left|\hat{W}(\tilde{H}+a)^{-1}f(x)\right|_x\leq  \left(\widehat{|W_1|}+\widehat{|W_2|}\right)(H_p+a)^{-1}|f|(x) \>\>\text{ for $\mu$-a.e. $x\in X$,} \nn
\end{align}
which implies that $\hat{W}(\tilde{H}+a)^{-1}$ is bounded, since $(\widehat{|W_1|}+\widehat{|W_2|})(H_p+a)^{-1}$ is compact, in particular bounded (the latter compactness follows from applying Theorem \ref{main} with $V_j:=|W_j|$, $V:=V_1+V_2$). Now one can directly use Theorem \ref{pits}.   
\end{proof}

\section{Covariant Schrödinger operators on weighted noncompact Riemannian manifolds}\label{Riemann}

\subsection{General facts} In this section we are going to specify the abstract measure theoretic results from the previous section to the setting of (weighted) Riemannian manifolds. Finally, in Example \ref{wasser} below, we are going to explain how the Hydrogen-type problems on manifolds that we have referred to in the introduction can be treated within our results.\vspace{1.2mm}

In section \ref{Riemann}, let $(M,g)$ be a connected smooth possibly noncompact Riemannian $m$-manifold (without boundary). The corresponding geodesic distance will be denoted $\Id_g(x,y)$, and the open geodesic ball at $x$ with radius $r$ will be written as 
$$
B_g(x,r)=\{y|\>\Id_g(x,y)<r\}.
$$
 Let $\mu_g$ stand for the Riemannian volume measure and for any $0<\varrho\in\mathsf{C}^{\infty}(M)$, let $\mu_{g,\varrho}$ be the smooth Borel measure 
$$
\Id\mu_{g,\varrho}( x)= \varrho(x)\Id\mu_g( x).
$$

If $E,F\to M$ are smooth (finite rank) Hermitian vector bundles, and 
$$
D:\Gamma_{\mathsf{C}^{\infty}}(M,E)\longrightarrow\Gamma_{\mathsf{C}^{\infty}}(M,F)
$$
is a smooth linear partial differential operator, then 
$$
D^{g,\varrho}:\Gamma_{\mathsf{C}^{\infty}}(M,F)\longrightarrow
\Gamma_{\mathsf{C}^{\infty}}(M,E) 
$$
will stand for its formal adjoint with respect to the complex scalar product  $\left\langle \cdot,\cdot\right\rangle_{\mu_{g,\varrho}}$. In other words, $D^{g,\varrho}:=D^{\dagger}$, where $\dagger$ refers to formal adjoints w.r.t. $\left\langle \cdot,\cdot\right\rangle_{\mu_{g,\varrho}}$.
\\
The symbol $\nabla^{TM}_g$ stands for the Levi-Civita connection w.r.t. $g$ on $T M$, where data corresponding to $T M$ will be considered as complexified, whenever necessary. In the above situation, the tripe $(M,g,\varrho)$ is usually referred to as a \emph{weighted Riemannian manifold} with weight function $\varrho$. The \emph{weighted Laplace-Beltrami operator} $\Delta_{g,\varrho}$ is defined by
\[
\Delta_{g,\varrho}:=\frac{1}{\varrho}\circ\mathrm{div}_g\circ \varrho\circ \mathrm{grad}_g:\mathsf{C}^{\infty}(M)\longrightarrow \mathsf{C}^{\infty}(M),
\]
with $\mathrm{div}_g$ the usual Riemannian divergence, and $\mathrm{grad}_g$ the Riemannian gradient.

Note that the weighted Laplace-Beltrami operator can also be written as 
$$
\Delta_{g,\varrho}=-\Id^{g,\varrho}\Id,
$$
where 
$$
\Id:\mathsf{C}^{\infty}(M)\longrightarrow \Gamma_{\mathsf{C}^{\infty}}(M, T^*M)
$$
stands for the usual exterior derivative, and where these data are complexified in the sequel. In particular, the operator $-\Delta_{g,\varrho}$ with domain of definition $\mathsf{C}^{\infty}_{\mathrm{c}}(M)$ is a symmetric nonnegative operator in $\mathsf{L}^2(M,\Id\mu_{g,\varrho})$. \\
The following result is by now well-known (cf. Theorem 7.13 and Theorem 9.5 in \cite{buch} for a detailed proof):

\begin{Propandef} For all fixed $y\in M$, there exists a pointwise minimal (necessarily smooth) element $p_{g,\varrho}(\cdot,\cdot,y)$ of the set given by all functions
$$
u:(0,\infty)\times M\longrightarrow [0,\infty)
$$
which satisfy the following equation in $(0,\infty)\times M$,
\[
\partial_t u(t,x)=\Delta_{g,\varrho} u(t,x),\>\>\lim_{t\to 0+} u(t,\cdot)=\delta_y. 
\]
The function  
\[
p_{g,\varrho}:(0,\infty)\times M\times M\longrightarrow  [0,\infty)
\]
is jointly smooth, and defines a pointwise consistent $\mu_{g,\varrho}$-heat-kernel in the sense of Definition \ref{heat}, called the \emph{minimal nonnegative heat kernel} on $(M,g,\varrho)$. 
\end{Propandef}

It follows that 
$$
H_{g,\varrho}:=H_{p_{g,\varrho}}
$$
is precisely the Friedrichs realization of $-\Delta_{g,\varrho}$ in $\mathsf{L}^2(M,\Id\mu_{g,\varrho})$ (cf. Corollary 4.11 in \cite{buch}). We refer the reader to \cite{buch} for proofs of the above facts.\vspace{1mm}

Likewise, given a smooth unitary (= Hermitian) covariant derivative $\nabla$ on a smooth Hermitian vector bundle $E\to M$, let $H^{\nabla}_{g,\varrho}$ denote the Friedrichs realization of the symmetric nonnegative operator $\nabla^{g,\varrho}\nabla$ in the Hilbert space $\Gamma_{\mathsf{L}^2}(M,E;\Id\mu_{g,\varrho})$, and given a locally integrable section $V\in\Gamma(M,\mathrm{End}(E))$ which is pointwise self-adjoint with nonnegative eigenvalues, let $H^{\nabla}_{g,\varrho,V}$ denote the form sum 
$$
H^{\nabla}_{g,\varrho,V}:=H^{\nabla}_{g,\varrho}\dotplus \hat{V}.
$$

\begin{Remark}\label{ddkja} In the usual Riemannian case $\varrho\equiv 1$, we will simply ommit the $\varrho$ everywhere in the notation that we have used above.
\end{Remark}

We record:

\begin{Lemma}\label{dlk} Given a smooth unitary covariant derivative $\nabla$ on a smooth Hermitian vector bundle $E\to M$, and a self-adjoint section $0\leq V\in\Gamma_{\mathsf{L}^1_{\mathrm{loc}}}(M,\mathrm{End}(E))$, one has the domination property $$H^{\nabla}_{g,\varrho,V}\succeq H_{g,\varrho}.$$
\end{Lemma}

\begin{proof} (Omitting $g$ and $\varrho$ in the notation) As in the unweighted case (see e.g. \cite{G1}), we can use a covariant Feynman-Kac formula for $\mathrm{e}^{-tH^{\nabla}_V}$: Let $(\mathbb{X}_t(x))_{t\geq 0}$ be a diffusion process which is generated by $-\Delta_{g,\varrho}$ (so $\varrho\equiv 1$ precisely corresponds to $g$-Brownian motion), issued to start from $x\in M$, let $\zeta(x)$ denote the random variable given by the lifetime of $X(x)$, and let
$$
\pa_t(x): E_{x}\longrightarrow E_{\mathbb{X}_t(x)},\>0\leq t<\zeta(x),
$$ 
denote the stochastic parallel transport w.r.t. $\nabla$ along the paths of $\mathbb{X}(x)$. The perturbation $V$ is taken into account as follows: It induces for $\mu$-a.e. $x\in M$ a pathwise linear process
$$
\IAA_t(x):E_{x}\longrightarrow E_{x},\>0\leq t<\zeta(x),
$$
given as the pathwise weak (= locally absolutely continuous) solution of
$$
\Id \IAA_t(x)=-\IAA_t(x)\left(\pa^{-1}_t(x)V(\mathbb{X}_t(x))\pa_t(x)\right)\Id t,\>\IAA_0(x)=\mathrm{id}_{E_{x}}.
$$
Then for all $t\geq 0$, $f\in\Gamma_{\mathsf{L}^2}(M,E;\Id\mu )$ and $\mu$-a.e. $x\in M$ one has
\begin{align}\label{gfd}
\mathrm{e}^{-tH^{\nabla}_V}f(x)=\mathbb{E}\Big[1_{\{t<\zeta(x)\}}\IAA_t(x)\pa^{-1}_t(x)f(\mathbb{X}_t(x)) \Big].
\end{align}
As $\nabla$ is unitary, $\pa_t(x)$ is almost surely unitary on the fibers, so that using $V\geq 0$ and Gronwall\rq{}s inequality one easily proves the almost surely valid bound
$$
\left|\IAA_t(x)\right|_{x}\leq 1.
$$
Thus, using once more that $\pa_t(x)$ is pathwise unitary, we arrive at the asserted estimate
$$
\left|\mathrm{e}^{-tH^{\nabla}_V}f(x)\right|_x\leq\mathbb{E}\Big[1_{\{t<\zeta(x)\}}|f(\mathbb{X}_t(x))|_{\mathbb{X}_t(x)}\Big]=\mathrm{e}^{-tH_p}|f|(x)=\int_M p(t,x,y) |f(y)|_y\Id\mu( y).
$$
The Feynman-Kac formula (\ref{gfd}) can be proved precisely as in \cite{G2}: Starting from the case $V$ bounded and smooth (cf. \cite{driver}), one first uses Friedrichs mollifiers to get the case $V$ bounded, and then a cut-off argument for the general case.
\end{proof}

\subsection{A relative compactness result for arbitrary weighted Riemannian manifolds}

The following definition is at the heart of this section:

\begin{Definition}\label{deee} Given $x\in M$, and $b_1,b_2>1$, let $r_{\mathrm{Eucl},g,\varrho}(x,b_1,b_2)$ be the supremum of all $r>0$ such that $B_g(x,r)$ is relatively compact and admits a chart with respect to which one has one has the following inequalities for all $y\in B_g (x,r)$, 
\begin{align}
&\f{1}{b_1}(\delta_{ij})\leq (g_{ij}(y)):= \left(g(\partial_i,\partial_j)(y)\right)\leq b_1 (\delta_{ij})\>\text{ as symmetric bilinear forms, and}\\
&\f{1}{b_2}\leq\varrho(y) \leq b_2. 
\end{align}
We call $r_{\mathrm{Eucl},g,\varrho}(x,b_1,b_2)$ \emph{the Euclidean radius of $ (M,g,\varrho)$ at $x$ with accuracy $(b_1,b_2)$.}
\end{Definition}

We collect some elementary properties of the Euclidean radius in the following 

\begin{Remark} It is clear that $r_{\mathrm{Eucl},g,\varrho}(x,b_1,b_2)\in (0,\infty]$, and it is cumbersome but elementary to check that the function
$$
M\longrightarrow (0,\epsilon],\>\>x\longmapsto \min(r_{\mathrm{Eucl},g,\varrho}(x,b_1,b_2),\epsilon) 
$$
is $1$-Lipschitz (w.r.t. to $g$), for every fixed $\epsilon>0$ (this follows from Proposition \ref{inj} in the appendix). In particular, 
$$
\inf_{x\in K}r_{\mathrm{Eucl},g,\varrho}(x,b_1,b_2)>0\>\>\text{ for every compact $K\subset M$.}
$$
\end{Remark}

The whole point of this definition is the following highly nontrivial fact, which is essentially due to A. Grigor'yan, and which relies on a parabolic $\mathsf{L}^2$-mean value inequality in combination with local Faber-Krahn inequalities:

\begin{Theorem}\label{gri} There exist a constant $C_1=C_1(m)>0$, which only depends on $m$ (and in particular not on $g$ or $\varrho$), such that for every $b_1,b_2>1$, $(t,x)\in (0,\infty)\times M$, $\epsilon_1>0, \epsilon_2>1$ one has
\begin{align*}
p_{g,\varrho}(t,x,x)\leq Cb_1^{m+4}b_2^{4/m+2}\min\big(t,\min(r_{\mathrm{Eucl},g,\varrho}(x,b_1,b_2),\epsilon_1)^2/\epsilon_2 \big)^{-m/2}.
\end{align*}
%\begin{align*}
%&p_{g,\varrho}(t,x,y)\leq\\ %&\frac{Cb_1^{m+4}b_2^{4/m+2}\left(1+\Id(x,y)^2/t\right)^{m/2}\exp\Big(-\Id(x,y)^2/(4t)-(t-t_0)\min\sigma(H_{p})\Big)}{\min\Big(t_0,\min(r_{\mathrm{Eucl}}(x,b_1,b_2),\epsilon_1)^2/\epsilon_2 \Big)^{m/4}\min\Big(t_0,\min(r_{\mathrm{Eucl}}(y,b_1,b_2),\epsilon_1)^2/\epsilon_2\Big )^{m/4}}\>\mid_{g,\varrho}.
%\end{align*}
\end{Theorem}

\begin{proof} (Omitting $g$ and $\varrho$ in the notation) The proof is based on Theorem 15.14 in \cite{buch} (see also Theorem 15.4 therein): Let $R:M\to (0,\infty)$ be any function with the following properties: Any (Riemannian!) ball $B(x,R(x))$ is relatively compact, and there is a number $a>0$ such that for any $x\in M$ and any open $U\subset B(x,R(x))$ one has the uniform Faber-Krahn inequality
$$
\min \sigma(H\mid_U)\geq a \mu(U)^{-2/m}.
$$
Then Theorem 15.14 in \cite{buch} implies the existence of a $C'=C'(m)>0$ which only depends on $m$, such that for all $t\geq t_0>0$, $x,y\in M$, one has
\begin{align*}
p(t,x,y)\leq \frac{C\rq{}\left(1+\Id(x,y)^2/t\right)^{m/2}\exp\left(-\Id(x,y)^2/(4t)-(t-t_0)\min\sigma(H )\right)}{a^2\min\big(t_0,R(x)^2 \big)^{m/4}\min\big(t_0,R(y)^2\big )^{m/4}}.
\end{align*}
We claim that $$R(x):=\min(r_{\mathrm{Eucl}}(x,b_1,b_2),\epsilon_1)/\sqrt{\epsilon_2}$$ is such a function: Indeed, given an open subset $U\subset B(x,R(x))$ one has (with $H_{\mathrm{Eucl}}$ the unweighted Euclidean Laplace operator, $\mu_{\mathrm{Eucl}}$ the usual unweighted Lebesgue measure, and $(g^{kl})=(g_{kl})^{-1}$)
\begin{align*}
\min \sigma(H\mid_U)&=\inf_{\psi\in\mathsf{C}_{\mathrm{c}}^{\infty}(U)}\int_U \sum_{k,l}g^{kl}\cdot\overline{\partial_k\psi}\cdot\partial_l\psi\cdot\varrho\cdot\sqrt{\mathrm{det}(g)}\Id \mu_{\mathrm{Eucl}}\\
&\geq \f{1}{b_2b_1^{m/2+1}}\min \sigma(H_{\mathrm{Eucl}}\mid_U)\geq \f{C''}{b_2b_1^{m/2+1}}\mu_{\mathrm{Eucl}}(U)^{-2/m}\\
&\geq \f{C'' }{  b_2^{2/m+1}b_1^{m/2+2}}\mu(U)^{-2m}=:a \mu(U)^{-2m},
\end{align*}
where the existence of $C''=C''(m)>0$ follows from the Euclidean Faber-Krahn inequality, and the proof is complete.
\end{proof}

Now we can formulate the main result of this section:

\begin{Corollary}\label{has} Let $\nabla$ be a unitary covariant derivative on the smooth Hermitian vector bundle $E\to M$, and let $V\in\Gamma_{\mathsf{L}^1_{\mathrm{loc}}}(M,\mathrm{End}(E))$ be pointwise self-adjoint with nonnegative eigenvalues. Assume further that $W\in\Gamma(M,\mathrm{End}(E))$ can be decomposed as $W=W_1+W_2$, where $W_j\in \Gamma(M,\mathrm{End}(E))$ are pointwise self-adjoint with $|W_2|\in\mathsf{L}^{\infty}(M,\Id\mu_{g,\varrho})$, and $W_1$ satisfies the following assumptions:
\begin{enumerate}
\item[$\cdot$] In case $m\leq 3$ assume that there exist $\epsilon>0$, $b_1,b_2>1$, such that for all $c>0$, one has 
\begin{align}\label{anna}
\int_M \big(|W_1|^{2} +1_{\{|W_2|>c\}}\big)    \min(r_{\mathrm{Eucl},g,\varrho}(\cdot,b_1,b_2),\epsilon)^{-m}    \Id\mu_{g,\varrho} <\infty.
\end{align}
\item[$\cdot$] In case $m\geq 4$ assume that there exist $b_1,b_2>1$, $q\rq{}> m/2$, such that 
\begin{align*}
\inf r_{\mathrm{Eucl},g,\varrho}(\cdot,b_1,b_2)>0,\>\>\int_M \big(|W_1|^{q\rq{}} +1_{\{|W_2|>c\}}\big)       \Id\mu_{g,\varrho}<\infty\>\text{ for all $c>0$}.
\end{align*}
\end{enumerate}
Then one has 
$$
\hat{W}(H^{\nabla}_{g,\varrho,V}+a)^{-1}\in\ILL^{\infty}\big(\Gamma_{\mathsf{L}^2}(M,E;\Id\mu_{g,\varrho})\big)\>\text{ for all $a>0$.} 
$$
%in particular, $\mathrm{Dom}(H^{\nabla}_{g,\varrho,V}+\hat{W})=\mathrm{Dom}(H^{\nabla}_{g,\varrho,V})$
%and $H^{\nabla}_{g,\varrho,V}+\hat{W}$ is self-adjoint and semibounded from below, with 
%$$
%\sigma_{\mathrm{ess}}(H^{\nabla}_{g,\varrho,V}+\hat{W})= \sigma_{\mathrm{ess}}(H^{\nabla}_{g,\varrho,V}).
%$$ 
\end{Corollary}

\begin{proof} Using Theorem \ref{gri} with $\epsilon_1=\epsilon$, $\epsilon_2=2$, implies the estimate
\begin{align*}
p_{g,\varrho}(t,x,x)\leq C\rq{}\min\big(t,\min(r_{\mathrm{Eucl},g,\varrho}(x,b_1,b_2),\epsilon)^2\big)^{-m/2},
\end{align*}
where $C\rq{}=C\rq{}(m,b_1,b_2)>0$. We will use Corollary \ref{mag1} with $\tilde{H}=H^{\nabla}_{g,\varrho,V}$ and appropriate choices of $F_j$ and $q$:\\
Case $m\leq 3$: We have
\begin{align*}
p_{g,\varrho}(t,x,x)\leq f_2(t)+F_1(x):=\frac{C\rq{}\rq{}}{t^{m/2}}+\frac{C\rq{}\rq{}}{\min(r_{\mathrm{Eucl},g,\varrho}(x,b_1,b_2),\epsilon)^m}.
\end{align*}
This can be estimated as follows
$$
f_2(t)+F_1(x)= F_1(x)(1+ f_2(t)/F_1(x))\leq F_1(x)\cdot\big(1+ (\inf F_1)f_2(t)\big)=:F_1(x)\cdot F_2(t),
$$
where clearly $\inf F_1>0$ by construction. Here the continuous function 
$$
F_2(t)^{1/2}=\big(1+ C\rq{}\rq{}\rq{}t^{-m/2}\big)^{1/2},\>\>t>0,
$$
is integrable near $0$ and bounded at $\infty$, so that $F_1$ is an $\IL^1$-control function, and (\ref{anna}) in combination with $\inf F_1>0$ implies $|W_2|\in\mathsf{L}_{\infty}(M,\Id\mu_{g,\varrho})$, so that the claim follows from using Corollary \ref{mag1} with $q=1$.\\
Case $m\geq 4$: In view of the assumption $\inf_{x\in M}r_{\mathrm{Eucl}}(x,b_1,b_2)>0$ we have the heat kernel estimate,
$$
p_{g,\varrho}(t,x,x)\leq \frac{C\rq{}}{\min(t,C\rq{}\rq{}\rq{})^{m/2}}=:F_2(t).
$$
Here for every $q>  m/4$ the continuous function $F_2(t)^{1/(2q)}$, $t>0$, is integrable near $0$ and bounded at $\infty$, so that $F_1\equiv 1$ is an $\mathsf{L}^q$-control function for every $q>m/4$ and the claim follows from using Corollary \ref{mag1} with $q=q\rq{}/2$. This completes the proof.
\end{proof}

We close this section with a very general Euclidean radius estimate of $(M,g,\varrho)$, which is not only of interest in connection with Theorem \ref{has}: In combination with Theorem \ref{gri}, Proposition \ref{below} below leads to entirely new weighted heat kernel estimates which do not require any global absolute bounds on the curvature:

\begin{Proposition}\label{below} Assume that there is a function $\beta:M\to (0,\infty)$ and a number $c_{\beta}<\infty$ such that for all $x,y\in M$,
\begin{align}\label{lipps}
\mathrm{Ric}_g(x)\geq - \beta(x)^{-2},\>\>\left|\beta(x)^{-2}-\beta(y)^{-2}\right| \leq c_{\beta}\Id_g(x,y) 
\end{align}
(in other words, the eigenvalues of $\mathrm{Ric}_g$ are pointwise bounded from below by a Lipschitz function). Then for any $b_1>1$ there exists a $C\rq{}=C\rq{}(m,b_1)>0$ such that for all $b_2>1$, $\delta>0$, $x\in M$, 
\begin{align*}
r_{\mathrm{Eucl},g,\varrho}(x,b_1,b_2)\geq C\rq{}  \min\Big(1\>,\> r_{\varrho}(x,b_2) \> , \> \min\big(\delta,r_{\mathrm{inj},g}(x)\big) \> , \>\  \beta(x)(1+c_{\beta})^{-1/2} \Big),
\end{align*}
where
$$
r_{g,\varrho}(x,b_2):=\sup\Big\{r\left|r>0, \> \text{ \emph{  $1/b_2\leq \varrho(y)\leq b_2$ for all $y\in B_g(x,r)$}}\Big\}\right. .
$$
\end{Proposition}

\begin{proof} As the $\delta$-capped injectivity radius is $1$-Lipschitz by Proposition \ref{inj} (appendix), we can use \cite[Proposition 2.5]{HPW} in combination with \cite[Proposition 2.3]{HPW}, to estimate for any $l>m$ the harmonic $\mathsf{W}^{1,l}$-Sobolev radius of $g$ with Euclidean distortion $b_1>1$, from below at each $x$ by
$$
r_{\mathrm{harm},g}(x,l,b_1)\geq C\rq{} \rq{} \min\Big(1 \ , \ \min\big(\delta,r_{\mathrm{inj},g}(x)\big)/2 \ , \  \beta(x)(1+c_{\beta})^{-1/2}\Big)
$$
for some $C\rq{}\rq{}=C\rq{}\rq{}(m,b_1,l)>0$, where we remark that the restriction $\beta \leq  1$ as well as $\beta,r\in\mathsf{C}^1(M)$ in the statement of \cite[Proposition 2.5]{HPW} are not used in its proof. What is used are the bounds (\ref{lipps}), and that $r(x):=\min\big(\delta,r_{\mathrm{inj},g}(x)\big)$ is Lipschitz w.r.t $g$.\\
Noting that for any $l$ one trivially has
$$
r_{\mathrm{Eucl},g}(x,b_1,b_2)=r_{\mathrm{Eucl},g,\varrho\rq{}}(x,b_1,b_2)\mid_{\varrho\rq{}\equiv 1}\>\geq r_{\mathrm{harm},g}(x,l,b_1)
$$
completes the proof.
\end{proof}

\subsection{A relative compactness result for weighted Riemannian manifolds with nonnegative weighted $\alpha$-Bakry-\'{E}mery tensor}

The purpose of this section is to show that under some lower curvature bounds, one can construct more explicit control functions for the heat kernel. To this end, we start with

\begin{Definition}\label{bakry}
Given $\alpha>0$, the \emph{$\alpha$-Bakry-\'{E}mery tensor of $(M,g,\varrho)$} is the smooth symmetric section in $T^*M\otimes T^* M\to M$, defined on smooth vector fields $A$,$B$ by
$$
\mathrm{Ric}_{g,\varrho,\alpha}(A,B):=\mathrm{Ric}_{g}(A,B)-\mathrm{Hess}_g[\log(\varrho)](A,B)-\f{1}{\alpha}\Id [\log(\varrho)](A)\cdot\Id[ \log(\varrho)](B),
$$
that is,
$$
\mathrm{Ric}_{g,\varrho,\alpha}=\mathrm{Ric}_{g}-\mathrm{Hess}_g[\log(\varrho)]-\f{1}{\alpha}\Id [\log(\varrho)]\otimes \Id[ \log(\varrho)],
$$
where $\Id[ \log(\varrho)]$ denotes the exterior differential of the function $\log(\varrho)$.
\end{Definition}

The name of $\mathrm{Ric}_{g,\varrho,\alpha}$ refers to the seminal paper \cite{bakry}, where a Ricci type curvature is associated to a general class of diffusion operators. Actually, the original definition given in \cite{bakry} corresponds in our situation to the \emph{$\infty$-Bakry-\'{E}mery tensor} which is given by 
$$
\mathrm{Ric}_{g,\varrho,\infty}(A,B):=\mathrm{Ric}_{g}(A,B)-\mathrm{Hess}_g[\log(\varrho)](A,B),
$$
while our definition of $\mathrm{Ric}_{g,\varrho,\alpha}$ is the one taken from \cite{qian} (see also \cite{lott}).\\
Now if one intends to generalize classical (= unweighted) Riemannian results that rely on nonnegative Ricci curvature to the weighted case, it turns out that for some results lower bounds on $\mathrm{Ric}_{g,\varrho,\infty}$ are enough, whereas for other results one needs lower bounds on $\mathrm{Ric}_{g,\varrho,\alpha}$ for some $\alpha<\infty$, a stronger assumption. This shows in particular that the weighted situation really leads to some mathematical subtleties. For example, it is shown in \cite{qian} that a weighted analog of the Myers\rq{} compactness result requires an assumptions of the form $\mathrm{Ric}_{g,\varrho,\alpha}\geq C>0$ for some positive $C$, $\alpha$ (in addition to the usual assumptions of completeness and connectedness) to hold true, and that indeed $\mathrm{Ric}_{g,\varrho,\infty}\geq C$ is not enough to conclude compactness. We refer the reader also to \cite{lott} for further (topological) investigations in this context, which rely on both, $\mathrm{Ric}_{g,\varrho,\alpha}$ and $\mathrm{Ric}_{g,\varrho,\infty}$.\\
Fur our purposes $\mathrm{Ric}_{g,\varrho,\alpha}$ is the more natural object, for it leads to Li-Yau type heat kernel bounds that do not require any absolute control on the derivative of the weight function (cf. \cite{chop} for a discussion of this). Namely, under geodesic completeness and $\mathrm{Ric}_{g,\varrho,\alpha}\geq 0$ for some $\alpha$, that is, $\mathrm{Ric}_{g,\varrho,\alpha}(A,A)\geq 0$ for all smooth vector fields $A$ on $M$, one can pick very explicit control functions:

\begin{Theorem}\label{ghhp} Assume that $(M,g)$ is geodesically complete and that $\mathrm{Ric}_{g,\varrho,\alpha}\geq 0$ for some $\alpha>0$. Then there exists a constant $C=C(m,\alpha)>0$ which only depends on $m$, $\alpha$, such that for all $\epsilon>0$, $(t,x)\in (0,\infty)\times M$ one has
$$
p_{g,\varrho}(t,x,x)\leq   C \mu_{g,\varrho}\big(B_g(x,\sqrt{t})\big)^{-1}\leq C\cdot \mu_{g,\varrho}\big(B_g(x,\epsilon)\big)^{-1}\cdot\big(F_{m,\alpha,\epsilon}(\sqrt{t}) +1\big),
$$
where for every $\beta\geq 0$, $R>0$, we have set
$$
F_{m,\beta,R}:(0,\infty)\longrightarrow (0,\infty),\>\>F_{m,\beta,R}(r):=2^{2m+2\beta}R^{m+\beta}r^{-(m+\beta)}.
$$
\end{Theorem}

\begin{proof} We start from the heat kernel bound 
$$
p_{g,\varrho}(t,x,y)\leq C_1(m,\alpha) \mu_{g,\varrho}\big(B_g(x,\sqrt{t})\big)^{-1}\mathrm{e}^{-\f{\Id_g(x,y)^2}{C_2(m,\alpha)t}},\>\>(t,x,y)\in (0,\infty)\times M\times M,
$$
which is a generalization of the classical Li-Yau bound for heat kernels with nonnegative Ricci tensor to the weighted case, and which can be found in \cite{chop}. The doubling property \cite{qian}
$$
\f{\mu_{g,\varrho}\big(B_g(x,2R)\big)}{\mu_{g,\varrho}\big(B_g(x,R)\big)}\leq 2^{m+\alpha},\>\> R>0,
$$
implies by a standard argument (cf. p.115 in \cite{saloff}) the doubling property
$$
\f{\mu_{g,\varrho}\big(B_g(x,R)\big)}{\mu_{g,\varrho}\big(B_g(x,r)\big)}\leq F_{m,\alpha,R}(r),\>\> R>r>0.
$$
Thus we have
$$
\mu_{g,\varrho}\big(B_g(x,\sqrt{r})\big)^{-1}\leq\mu_{g,\varrho}\big(B_g(x,\sqrt{\epsilon})\big)^{-1} F_{m,\alpha,\sqrt{\epsilon}}(\sqrt{r}),\>\>0<r<\epsilon,
$$
so that
\begin{align*}
&p_{g,\varrho}(t,x,x)\leq  CF_{m,\alpha,\sqrt{\epsilon}}(\sqrt{t}) \cdot\mu_{g,\varrho}\big(B_g(x,\sqrt{\epsilon})\big)^{-1}  + C \mu_{g,\varrho}\big(B_g(x,\sqrt{\epsilon})\big)^{-1}\\
&= C\cdot \mu_{g,\varrho}\big(B_g(x,\sqrt{\epsilon})\big)^{-1}\cdot\Big(F_{m,\alpha,\sqrt{\epsilon}}(\sqrt{t}) +1\Big),
\>\>(t,x)\in (0,\infty)\times M,
\end{align*}
which completes the proof.
\end{proof}

This result implies the following relative compactness criterion:

\begin{Corollary}\label{has2} Assume that $(M,g)$ is geodesically complete and that $\mathrm{Ric}_{g,\varrho,\alpha}\geq 0$ for some $\alpha>0$. Let $\nabla$ be a unitary covariant derivative on the smooth Hermitian vector bundle $E\to M$, and let $V\in\Gamma_{\mathsf{L}^1_{\mathrm{loc}}}(M,\mathrm{End}(E))$ be pointwise self-adjoint with nonnegative eigenvalues. Assume further that $W\in\Gamma(M,\mathrm{End}(E))$ can be decomposed as $W=W_1+W_2$, where $W_j\in \Gamma(M,\mathrm{End}(E))$ are pointwise self-adjoint with $|W_2|\in\mathsf{L}^{\infty}(M,\Id\mu_{g,\varrho})$, and $W_1$ satisfies the following assumptions:
\begin{enumerate}
\item[$\cdot$] In case $m+\alpha<4$ assume that there exists an $\varepsilon>0$, such that for all $c>0$, one has 
\begin{align}\label{anna3}
\int_M \big(|W_1 |^{2}\mu_{g,\varrho}\big(B_g(\cdot,\varepsilon)\big)^{-1}  +1_{\{|W_2|>c\}} \mu_{g,\varrho}\big(B_g(\cdot,\varepsilon)\big)^{-1} +1_{\{|W_2|>c\}}\big)      \Id\mu_{g,\varrho} <\infty.
\end{align}
\item[$\cdot$] In case $m+\alpha\geq 4$, assume that there exist $\varepsilon>0$,
 $q\rq{}>(m+\alpha)/4$, such that 
\begin{align*}
\inf\mu_{g,\varrho}\big(B_g(\cdot,\varepsilon)\big)>0,\>\>\int_M \big(|W_1|^{q\rq{}} +1_{\{|W_2|>c\}}\big)       \Id\mu_{g,\varrho}<\infty\>\text{ for all $c>0$}.
\end{align*}
\end{enumerate}
Then one has 
$$
\hat{W}(H^{\nabla}_{g,\varrho,V}+a)^{-1}\in\ILL^{\infty}\big(\Gamma_{\mathsf{L}^2}(M,E;\Id\mu_{g,\varrho})\big)\>\text{ for all $a>0$.} 
$$
%in particular, $\mathrm{Dom}(H^{\nabla}_{g,\varrho,V}+\hat{W})=\mathrm{Dom}(H^{\nabla}_{g,\varrho,V})$, and $H^{\nabla}_{g,\varrho,V}+\hat{W}$ is self-adjoint and semibounded from below, with $\sigma_{\mathrm{ess}}(H^{\nabla}_{g,\varrho,V}+\hat{W})= \sigma_{\mathrm{ess}}(H^{\nabla}_{g,\varrho,V})$. 
\end{Corollary}

\begin{proof} Referring to the statement of Theorem \ref{ghhp}, we set for $t>0$,
$$
F_2(t):=C(m,\alpha)F_{m,\alpha,\varepsilon}(\sqrt{t})+C(m,\alpha).
$$
Note that $F_2(t)$ is continuous, behaves like $t^{m/2+\alpha/2}$ at $0$, and is bounded at $\infty$.\\
Case $m+\alpha<4$: Here the continuous function $F_{2}^{1/2}(t)$ is integrable near $0$ and bounded at $\infty$, so that 
$$
F_1(x):=\mu_{g,\varrho}\big(B_g(x,\varepsilon)\big)^{-1} 
$$
is an $\mathsf{L}^1$-control function, and the claim follows immediately from using Corollary \ref{mag1} with $q=1$.\\
Case $m+\alpha\geq 4$: Here the continuous function $F_{2}^{1/(2q)}(t)$ is integrable at $0$ for all $q>m/4+\alpha/4$ and bounded at $\infty$, so that $F_1(x)\equiv 1$ is an $\mathsf{L}^q$-control function for these values of $q$, and using Corollary \ref{mag1} with $q=q\rq{}/2$ proves the claim.
\end{proof}

\subsection{A modified relative compactness result for unweighted Riemannian manifolds with nonnegative Ricci curvature}\label{dhjkp}

If one has $\mathrm{Ric}_g\geq 0$ in the usual Riemannian (that is, unweighted) situation, there hold sharper statements which are very close to the Euclidean case, as one can take $\alpha\to 0$ in the results of the last section, in a certain sense to be made precise in the sequel. Reminding the reader of our convention concerning the notation in the unweighted case (cf. Remark \ref{ddkja}), we give ourselves a geodesically complete Riemannian manifold $(M,g)$ with $\mathrm{Ric}_g\geq 0$. Then by the usual Li-Yau heat kernel estimate there exists constants $C_j=C_j(m)>0$ which only depend on $m$, such that for all $(t,x,y)$ one has
\begin{align}\label{hiop}
p_{g}(t,x,y)\leq   C_1 \mu_{g}\big(B_g(x,\sqrt{t})\big)^{-1}\mathrm{e}^{-\f{\Id_g(x,y)^2}{C_2t}}.
\end{align}
In this case, one has the usual Euclidean doubling property (cf. Theorem 5.6.4 in \cite{saloff})
$$
\f{\mu_{g}\big(B_g(x,R)\big)}{\mu_{g}\big(B_g(x,r)\big)}\leq 2^{2m}R^{m/2}r^{-m/2},\>\> R>r>0,
$$
so that the same argument as in the proof of Theorem \ref{ghhp} yield the heat kernel bound
$$
p_{g}(t,x,x)\leq   C \mu_{g}\big(B_g(x,\sqrt{t})\big)^{-1}\leq C\cdot \mu_{g}\big(B_g(x,\epsilon)\big)^{-1}\cdot\big(F_{m,0,\epsilon}(\sqrt{t}) +1\big),
$$
for each $\epsilon>0$, which sharpens the weighted result. In addition we now have a uniform lower bound on the control function, in the sense that
$$
\inf_{x\in M} \mu_{g}\big(B_g(x,\varepsilon)\big)^{-1}=\sup_{x\in M} \mu_{g}\big(B_g(x,\varepsilon)\big)\leq C(m) \varepsilon^m,\>\text{ for all $\varepsilon>0$,}
$$
so that 
$$
\mathsf{L}_{\infty}\big(M \, , \, \mu_{g}\big(B_g(\cdot,\varepsilon)\big)^{-1}\Id\mu_g\big)\subset \mathsf{L}_{\infty}(M,\Id\mu_g).
$$
Now the same arguments as in the proof of Corollary \ref{has2} yield:

\begin{Corollary}\label{has3} Assume that $(M,g)$ is geodesically complete with $\mathrm{Ric}_{g}\geq 0$. Let $\nabla$ be a unitary covariant derivative on the smooth Hermitian vector bundle $E\to M$, and let $V\in\Gamma_{\mathsf{L}^1_{\mathrm{loc}}}(M,\mathrm{End}(E))$ be pointwise self-adjoint with nonnegative eigenvalues. Assume further that $W\in\Gamma(M,\mathrm{End}(E))$ can be decomposed as $W=W_1+W_2$, where $W_j\in \Gamma(M,\mathrm{End}(E))$ are pointwise self-adjoint with $|W_2|\in\mathsf{L}^{\infty}(M,\Id\mu_{g})$, and $W_1$ satisfies the following assumptions:
\begin{enumerate}
\item[$\cdot$] In case $m\leq 3$ assume that there exists an $\varepsilon>0$ such that for all $c>0$, one has 
\begin{align}\label{anna2}
\int_M \big(|W_1 |^{2}  +1_{\{|W_2|>c\}}  \big)     \mu_{g}\big(B_g(\cdot,\varepsilon)\big)^{-1} \Id\mu_{g} <\infty.
\end{align}
\item[$\cdot$] In case $m\geq 4$ assume that there exist $\varepsilon>0$, $q\rq{}>m/4$ such that 
\begin{align*}
\inf\mu_{g}\big(B_g(\cdot,\varepsilon)\big)>0,\>\>\int_M \big(|W_1|^{q\rq{}} +1_{\{|W_2|>c\}}\big)       \Id\mu_{g}<\infty\>\text{ for all $c>0$}.
\end{align*}
\end{enumerate}
Then one has 
$$
\hat{W}(H^{\nabla}_{g,V}+a)^{-1}\in\ILL^{\infty}\big(\Gamma_{\mathsf{L}^2}(M,E;\Id\mu_{g})\big)\>\text{ for all $a>0$.} 
$$
%in particular, $\mathrm{Dom}(H^{\nabla}_{g,\varrho,V}+\hat{W})=\mathrm{Dom}(H^{\nabla}_{g,\varrho,V})$, and $H^{\nabla}_{g,\varrho,V}+\hat{W}$ is self-adjoint and semibounded from below, with $\sigma_{\mathrm{ess}}(H^{\nabla}_{g,\varrho,V}+\hat{W})= \sigma_{\mathrm{ess}}(H^{\nabla}_{g,\varrho,V})$. 
\end{Corollary}

\begin{Remark}\label{michel} Comparing Corollary \ref{has3} with its weighted variant Corollary \ref{has2}, we see that the parameter $\alpha>0$ in the curvature assumption $\mathrm{Ric}_{g,\varrho,\alpha}\geq 0$ plays the role of a \lq\lq{}virtual\rq\rq{} dimension. This is also reflected in (in fact: implied by) the corresponding heat kernel bounds. We have borrowed this terminology from Michele Rimoldi\rq{}s PhD-thesis \cite{rimoldi}, where this aspect has been investigated in the context of geometric rigidity results.
\end{Remark}

We close this section with the following example, that makes contact with the Hydrogen type problems on nonparabolic Riemannian $3$-manifolds that have been considered in the introduction. In fact it deals with a more general situation, taking magnetic fields into account:

\begin{Example}\label{wasser} Let $(M,g)$ be geodesically complete with $m=3$ and $\mathrm{Ric}_g\geq 0$. Let $\nabla$ be a  Hermitian covariant derivative on the smooth Hermitian vector bundle $E\to M$, and consider potentials
$$
0\leq V\in\Gamma_{\mathsf{L}^1_{\mathrm{loc}}}(M,\mathrm{End}(E)), \>V'\in\Gamma_{\mathsf{L}^2}(M,\mathrm{End}(E);\Id\mu_g).
$$
In this situation, the additional lower Euclidean volume growth assumption 
$$
\inf_{x\in M} \mu_g(B_g(x,r))r^{-m}>0
$$
implies that $(M,g)$ is nonparabolic (this follows readily from (\ref{hiop})), that is, for $x\ne x_0$, one has
$$
G(x,x_0):=\int^{\infty}_0 p_g(t,x,y)\Id t<\infty.
$$
In fact, $G(x,x_0)\leq C \Id_g(x,x_0)$, which entails that 
$$
W_{V',\kappa,g,x_0}:=V'-\kappa G(\cdot,x_0)=\big(V'-\kappa 1_{B_g(x_0,1)}G(\cdot,x_0)\big) -\kappa 1_{M\setminus B_g(x_0,1)}G(\cdot,x_0)=:W_1-W_2=:W
$$
satisfies the assumptions from Corollary \ref{has3}, for every $\kappa>0$. Thus,
$$
\widehat{W_{V',\kappa,g,x_0}} (H^{\nabla}_{g,V}+a)^{-1}\in\ILL^{\infty}\big(\Gamma_{\mathsf{L}^2}(M,E;\Id\mu_{g})\big)\>\text{ for all $a>0$.}
$$ 
As a particular case of this construction we can take the trivial complex line bundle $E=M\times \IC\to M$ with $\nabla=\Id+\sqrt{-1}\beta$, where $\beta$ is a smooth real-valued $1$-form on $M$. Then, picking furthermore $V=V'=0$, shows that for all fixed $\kappa>0$, $x_0\in M$, one has 
$$
\widehat{-\kappa G(\cdot,x_0)} (H^{\beta}_g+a)^{-1}\in\ILL^{\infty}(\mathsf{L}^2(M,\Id\mu_{g})),
$$
where $H^{\beta}_g$ denotes the Friedrichs realization in $\mathsf{L}^2(M,\Id\mu_{g})$ of 
\begin{align*}
-\Delta^{\beta}_g\Psi&=\left(\Id+\sqrt{-1}\beta\right)^{g}\left(\Id+\sqrt{-1}\beta\right)\Psi\\
&= -\Delta_g\Psi-2\sqrt{-1} \ g^*(\beta,\Id\Psi)+\big(\sqrt{-1} \Id^{g}\beta+ |\beta|^2_{g^*}\big)\Psi,\>\>\>\Psi\in\mathsf{C}^{\infty}_{\mathrm{c}}(M),
\end{align*}
which is the unique self-adjoint realization of the latter operator, as we assume $(M,g)$ to be complete. The operator $H^{\beta}_g-\widehat{\kappa G(\cdot,x_0)}$ is thus a well-defined self-adjoint operator in $ \mathsf{L}^2(M,\Id\mu_{g})$, which is in fact essentially self-adjoint on $\mathsf{C}^{\infty}_{\mathrm{c}}(M)$, and which can be interpreted as the Hamilton operator of an electron in the magnetic field $\Id\beta$ and in the electric potential of a nucleus which is considered to be located in $x_0$, having 
$\sim \kappa$ protons. Another consequence of the the above relative compactness is that the operators $H^{\beta}_g$ and $H^{\beta}_g-\widehat{\kappa G(\cdot,x_0)}$ have the same essential spectrum.
\end{Example}

\section{Covariant Schrödinger operators on infinite weighted graphs}\label{grapp}

In this section we are going to apply our abstract result measure space results to weighted discrete problems. To this end, we start with:

\begin{Definition} A \emph{weighted graph} is a triple $(X,b,\varrho)$, where $X$ is a countable set, $b$ is a symmetric function
$$
b: X\times X\longrightarrow  [0,\infty) \text{ with $b(x,x)=0$, $\sum_{y\in X} b(x,y)<\infty$ for all $x\in X$,}
$$
and $\varrho:X\to (0,\infty)$ is an arbitrary function. 
\end{Definition}

\emph{For the rest of this section, we fix an arbitrary weighted graph $(X,b,\varrho)$.} \vspace{1.4mm}

In this context, $b$ is interpreted as an edge weight function and one writes $x\sm y$, if $b(x,y)>0$. In other words, $X$ is understood to be the set of vertices of the graph, and the set $\{b>0\}$ is interpreted as the set of weighted edges of the graph (where $b(x,x)=0$ means that we avoid loops). Note that we allow each vertex to have infinitely many neighbours, which means that we can treat graphs that need not be locally finite.\\
We consider $X$ as being equipped with the sigma-algebra $2^{X}$, so that the vertex weight function $\varrho$ defines a measure $\mu_{\varrho}$ in the obvious way:
$$
\mu_{\varrho}(A)=\sum_{x\in A}\varrho(x), \>\>A\subset X.
$$ 

\emph{We are going to assume in the sequel that $(X,b)$ is connected, in the usual graph-theoretic sense that for any $x,y\in X$ there is a finite sequence $x_1,\dots,x_n\in X$ such that $x_0=x$, $x_n=y$.}\vspace{1.5mm}

The space of complex-valued functions on $X$ will be denoted with $\mathsf{C}(X)$, where an index \lq{}$\mathrm{c}$\rq{} now simply means \lq{}finitely supported\rq{}. We define a set
$$
\mathsf{F}_{b}(X):=\left\{\psi\left|\psi\in\mathsf{C}(X), \sum_{y\in X}b(x,y)|\psi(y)|<\infty\>\text{ for all $x\in X$}\right\}\right.\supset \mathsf{L}^{\infty}(X),
$$
and a formal difference operator $\Delta_{b,\varrho}$ by
\begin{align}\label{formal}
\Delta_{b,\varrho}:\mathsf{F}(X)\longrightarrow \mathsf{C}(X),\>\Delta_{b,\varrho}\psi(x)=-\f{1}{\varrho(x)}\sum_{\{y|y\sm x\}}b(x,y)\big(\psi(x)-\psi(y)\big).
\end{align}

Using a discrete maximum principle, one can deduce:

\begin{Propandef}\label{dd} For all fixed $y\in X$, there exists a pointwise minimal element $p_{b,\varrho}(\cdot,\cdot,y)$ of the set given by all bounded functions 
$$
u:[0,\infty)\times X\to [0,\infty)
$$
that satisfy the following equation in $[0,\infty)\times X$,
$$
\partial_t u(t,x)=\Delta_{b,\varrho} u(t,x),\>\> u(0,x)=\delta_y(x).
$$
The function  
\[
p_{b,\varrho}:[0,\infty)\times X\times X\longrightarrow  [0,\infty)
\]
induces a pointwise consistent $\mu_{\varrho}$-heat-kernel in the sense of Definition \ref{heat}, called the \emph{minimal nonnegative heat kernel} on $(X,b,\varrho)$. 
\end{Propandef}

\begin{proof} This follows from combining Theorem 10, Lemma 5.1 and Theorem 11 in \cite{kl}. The set of uniqueness can certainly be enlarged, if necessary. As we only intended to define $p$ analogously to the manifold case, we have not worked into this direction.
\end{proof}

This result entails a first fundamental difference to the manifold case: the discrete weighted heat kernel is differentiable in time up to $t=0$ (for fixed $x,y$).\\
Let us identify the operator $H_{p}$ in this case: Define first a symmetric sesqui-linear form $\tilde{Q}_b$ with domain of definition $\mathsf{C}_{\mathrm{c}}(X)$ by
\begin{align*}
\tilde{Q}_b(\psi_1,\psi_2):= & \frac{1}{2}\sum_{x\sim y} b(x,y)\overline{\big(\psi_1(x)-\psi_1(y)\big)}\big(\psi_2(x)- \psi_2(y)\big).
\end{align*}
Clearly, $\tilde{Q}_b$ is densely defined and nonnegative in $\mathsf{L}^2(X,\Id \mu_{\varrho})$, and in fact it is closable. Note that the scalar product $\left\langle \cdot,\cdot\right\rangle_{\mu_{\varrho}}$ is now simply given by
$$
\left\langle f_1,f_2\right\rangle_{\mu_{\varrho}}=\sum_{x\in X}\overline{f_1(x)}f_2(x)\varrho(x).
$$
However, in contrast to the Riemannian setting, $\tilde{Q}_b$ need not come from a symmetric operator! The reason for the latter fact is simply that any such symmetric operator must necessarily be a restriction of $-\Delta_{b,\varrho}$, but depending on the global geometry of $(X,b,\varrho)$, $\Delta_{b,\varrho}$ obviously need not map $\mathsf{C}_{\mathrm{c}}(X)$ into $\mathsf{L}^2(X,\Id \mu_{\varrho})$ (cf. formula (\ref{formal})).\\
Furthermore, we point out that in general $\tilde{Q}_b$ is not bounded. Nevertheless one always has the bound 
\begin{align}\label{scal}
\tilde{Q}_b(\psi,\psi)\leq 2 C(b,\varrho) \|\psi\|^2_{\mu_{\varrho},2},\>\text{where $C(b,\varrho):=\sup_{x\in X} \frac{1}{\varrho(x)}\sum_{y\in X} b(x,y)\in [0,\infty]$,}
\end{align}
which however entails that in many applications actually is bounded (for example, on the usual unweighted lattice $\IZ^m$). The self-adjoint operator corresponding to the closure of $\tilde{Q}_b$ is precisely $H_{p_{b,\varrho}}$.\\
A measurable vector bundle $E\to X$ with rank $d$ is now nothing but a a family $E=\bigsqcup_{x\in X}E_x\to X$ of $d$-dimensional complex linear spaces, and a Hermitian structure on such a bundle is nothing but a family of complex scalar products on each fiber.\\
The following definition is borrowed from combinatorics \cite{Kenyon-11}, where it is used in the context of generalized matrix-tree theorems (see also \cite{GMT} for a corresponding covariant Feynman-Kac formula):

\begin{Definition} Let $E\to X$ be a complex vector bundle with rank $d$.\\
(i) An assignment $\Phi$ which assigns to any $x\sm y$ an isomorphism of complex vector spaces $\Phi_{x,y}: E_x\to E_y$ is called a \emph{$b$-connection} on $E\to X$, if one has $\Phi_{y,x}=\Phi^{-1}_{x,y}$ for all $x\sm y$. \\
(ii) If $E\to X$ is Hermitian, then a $b$-connection $\Phi$ on $E\to X$ is called \emph{unitary}, if $\Phi_{x,y}^*=\Phi^{-1}_{x,y}$ for all $x\sm y$.
\end{Definition}

For the moment, we fix a Hermitian vector bundle $E\to X$ of rank $d$, with a unitary $b$-connection $\Phi$ defined on it. As in the scalar case (\ref{scal}), these data determine the symmetric sesquilinear form $\tilde{Q}^{\Phi}_b$ given by 
\begin{align*}
\tilde{Q}^{\Phi}_b(\psi_1,\psi_2)=&\>\frac{1}{2}\sum_{x\sm y}b(x,y)\big(\psi_1(x)-\Phi_{y,x} \psi_1(y),\psi_2(x)-\Phi_{y,x} \psi_2(y)\big)_{x}
\end{align*}
on the domain of definition $\Gamma_{c}(X,F)$. Again, this form is densely defined, nonnegative and closed in $\Gamma_{\mathsf{L}^2}(X,F;\Id \mu_{\varrho} )$, and in general it is unbounded (again an upper bound is given by $2 C(b,\varrho)\in [0,\infty]$). We remark that on discrete bundles $\left\langle \cdot,\cdot\right\rangle_{\mu_{\varrho}}$ is now given by
$$
\left\langle f_1,f_2\right\rangle_{\mu_{\varrho}}=\sum_{x\in X}(f_1(x),f_2(x))_x\varrho(x),\>\>f_j\in\Gamma_{\mathsf{L}^2}(X,F;\Id \mu_{\varrho} ).
$$

The self-adjoint operator in corresponding to the closure of the above form will be denoted with $H^{\Phi}_{b,\varrho}$. We refer the reader to \cite{G45} for proofs of the above facts and more details on these covariant operators (noting that \cite{G45} deals with semiclassical limits of the corresponding Schrödinger semigroups). \\
Given a section $V\in\Gamma(X,\mathrm{End}(E))$ which is pointwise self-adjoint with nonnegative eigenvalues, let $H^{\Phi}_{b,\varrho,V}$ denote the form sum 
$$
H^{\Phi}_{b,\varrho,V}:=H^{\Phi}_{b,\varrho}\dotplus \hat{V}.
$$

As in the Riemannian case of Lemma \ref{dlk}, $H_{b,\varrho}:=H_{p_{b,\varrho}}$ generates a Markoff process, and one can use path integral techniques to prove the domination property (cf. Theorem 2 in \cite{GMT})
$$
H^{\Phi}_{b,\varrho,V}\succeq H_{b,\varrho}.
$$
However, a second fundamental difference to the continuum case is the following uniform heat kernel estimate, again valid without any further assumptions on $(X,b,\varrho)$:
\begin{align}\label{dzzu}
p_{b,\varrho}(t,x,y)\leq 1/\varrho(y)\>\text{ for all $t\geq 0$, $x,y\in X$}.
\end{align}
The probably most intuitive way to understand (\ref{dzzu}) is to note that, as we have already stated, $H_{b,\varrho}$ generates a Markoff process $(\mathbb{X}^{b,\varrho}_t(x))_{t\geq 0}$, and the probability of finding the underlying Markoff particle at time $t$ in $A\subset X$, when conditioned to start in $x\in X$, is precisely the quantity
$$
\mathbb{P} (\mathbb{X}^{b,\varrho}_t(x)\in A )=\sum_{z\in A}p_{b,\varrho}(t,x,z)\varrho(z)\leq 1.
$$
In contrast to the continuum setting, the set $A=\{y\}$ does not have zero measure, and we end up with (\ref{dzzu}). \\
After these preparations, the following result becomes a simple consequence of Corollary \ref{mag1}:

\begin{Theorem}\label{grap} Let $\Phi$ be a unitary $b$-connection on the Hermitian vector bundle $E\to X$, and let $V\in\Gamma(X,\mathrm{End}(E))$ be pointwise self-adjoint with nonnegative eigenvalues. Assume that $W\in\Gamma(X,\mathrm{End}(E))$ can be decomposed as $W=W_1+W_2$, where $W_j\in \Gamma(X,\mathrm{End}(E))$ are pointwise self-adjoint with 
\begin{align*}
\sum_{x\in X} |W_1(x)|^2 \varrho(x)+\sum_{x\in\{|W_2|>c\}} (  1+   \varrho(x) )  <\infty\text{ for all $c>0$.}
\end{align*}
Then one has 
$$
\hat{W}(H^{\Phi}_{b,\varrho,V}+a)^{-1}\in\ILL^{\infty}\big(\Gamma_{\mathsf{L}^2}(X,E;\Id\mu_{\varrho})\big)\>\text{ for all $a>0$.} 
$$
%In particular, one has 
%$$
%\mathrm{Dom}(H^{\Phi}_{b,\varrho,V}+\hat{W})=\mathrm{Dom}(H^{\Phi}_{b,\varrho,V})
%$$
% and $H^{\Phi}_{b,\varrho,V}+\hat{W}$ is self-adjoint and semibounded from below, with 
%$$
%\sigma_{\mathrm{ess}}(H^{\Phi}_{b,\varrho,V}+\hat{W})= \sigma_{\mathrm{ess}}(H^{\Phi}_{b,\varrho,V}).
%$$ 
\end{Theorem}

\begin{proof} Taking $F_1=0$, $F_2=1/\varrho$, everything follows from using Corollary \ref{mag1} with $\tilde{H}=H^{\Phi}_{b,\varrho,V}$, noting that in the discrete case, the assumption
$$
\sum_{x\in\{|W_2|>c\}} (  1+   \varrho(x) )  <\infty\text{ for all $c>0$}
$$
implies the boundedness of $|W_2|$.
\end{proof}

\appendix
\section{A result on the Lipschitz continuity of injectivity radius-type functions}

\begin{Proposition}\label{inj} Let $X\equiv (X,\Id)$ be a metric space, and let 
$$
\IPP: X\times (0,\infty) \longrightarrow \{0,1\}
$$
be a map (considered to be \lq\lq{}a property of metric balls in $X$\rq\rq{}) which satisfies the following assumption: If $x\in X$, $r>0$ are such that $\IPP(x,r)=1$, then one also has $\IPP(y,s)=1$ for all $y,s$ with $0<s< r-\Id(x,y)$. Then for any $\epsilon>0$ the map
$$
R_{\epsilon}: X\longrightarrow [0,\epsilon],\> R_{\epsilon}(x):=\min\Big(\sup\{r\left|r>0,\IPP(x,r)=1  \}\right. , \epsilon\Big)
$$
is $1$-Lipschitz.
\end{Proposition}

\begin{proof} This can be proved precisely as Lemma 2.3 in \cite{BGM}. Note that the assumption on the property $\IPP$ simply means the following: If $\IPP$ is true on an open ball $B(x,r)=\{z| \ \Id(x,z)<r\}$, then $\IPP$ is true for any open ball $B(y,s)\subset B(x,r)$. The supremum in the definition of the underlying radius-type  function $R_{\infty}$ has to be capped simply in order to make this quantity finite (otherwise it becomes ambiguous to speak about Lipschitz continuity).
\end{proof}

\vspace{3mm}

{\bf Acknowledgements:} The second named author (B.G.) has been financially supported by the SFB 647: Raum - Zeit - Materie.

\end{document}